\documentclass[reqno]{amsart}
\usepackage{amsthm,amsmath,amssymb}
\usepackage{stmaryrd,mathrsfs}
\usepackage[OT2,T1]{fontenc}
\usepackage{esint}
\usepackage{color}

\allowdisplaybreaks

\newcommand{\ud}[0]{\,\mathrm{d}}

\newcommand{\abs}[1]{|#1|}
\newcommand{\Babs}[1]{\Big|#1\Big|}

\newcommand{\Norm}[2]{\|#1\|_{#2}}

\newcommand{\BNorm}[2]{\Big\|#1\Big\|_{#2}}
\newcommand{\pair}[2]{\langle #1,#2 \rangle}

\newcommand{\ave}[1]{\langle #1\rangle}

\newcommand{\BMO}[0]{\operatorname{BMO}}

\renewcommand{\Re}[0]{\operatorname{Re}}

\newcommand{\R}{\mathbb{R}}

\newcommand{\Z}{\mathbb{Z}}

\newcommand{\Exp}[0]{\mathbb{E}}

\newcommand{\eps}[0]{\varepsilon}

\newcommand{\strt}[1]{\rule{0pt}{#1}}

\newcommand{\al}{\alpha}
\newcommand{\be}{\beta}

\newcommand{\ez}{\epsilon}

\newcommand{\la}{\lambda}

\newcommand{\si}{\sigma}

\newcommand{\vp}{\varphi}

\newcommand{\Om}{\Omega}

\def\Xint#1{\mathchoice
  {\XXint\displaystyle\textstyle{#1}}%
  {\XXint\textstyle\scriptstyle{#1}}%
  {\XXint\scriptstyle\scriptscriptstyle{#1}}%
  {\XXint\scriptscriptstyle\scriptscriptstyle{#1}}%
  \!\int}
\def\XXint#1#2#3{{\setbox0=\hbox{$#1{#2#3}{\int}$}
    \vcenter{\hbox{$#2#3$}}\kern-.5\wd0}}

\def\avgint{\Xint-}

\def\cyr{\fontencoding{OT2}\fontfamily{wncyr}\selectfont}
\DeclareTextFontCommand{\textcyr}{\cyr}
\newcommand{\sha}[0]{\textup{\textcyr{SH}}}

\swapnumbers \numberwithin{equation}{section}

\theoremstyle{plain}
\newtheorem{theorem}[equation]{Theorem}
\newtheorem{proposition}[equation]{Proposition}
\newtheorem{corollary}[equation]{Corollary}
\newtheorem{lemma}[equation]{Lemma}

\theoremstyle{definition}

\theoremstyle{remark}
\newtheorem{remark}[equation]{Remark}

\makeatletter
\@namedef{subjclassname@2010}{%
  \textup{2010} Mathematics Subject Classification}
\makeatother

\begin{document}

\title{Sharp weighted bounds involving $A_{\infty}$}

\author[T.~Hyt\"onen]{Tuomas Hyt\"onen}
\address{Department of Mathematics and Statistics, P.O.B.~68 (Gustaf H\"all\-str\"omin katu~2b), FI-00014 University of Helsinki, Finland}
\email{tuomas.hytonen@helsinki.fi}

\author[C.~P\'erez]{Carlos P\'erez}
\address{
Departamento de An\'alisis Matem\'atico, Facultad de Matem\'aticas,
Universidad De Sevilla, 41080 Sevilla, Spain
}
\email{carlosperez@us.es}



\thanks{T. H. was supported by the Academy of Finland, projects 130166, 133264 and 218148.}
\thanks{C. P. was supported by the Spanish Ministry of Science and Innovation grant MTM2009-08934 and by the Junta de Andaluc\'ia, grant FQM-4745}

\maketitle

\begin{abstract}
We improve on several weighted inequalities of recent interest by replacing a part of the $A_p$ bounds by weaker $A_\infty$ estimates
involving Wilson's $A_\infty$ constant
\begin{equation*}
  [w]_{A_\infty}':=\sup_Q\frac{1}{w(Q)}\int_Q M(w\chi_Q).
\end{equation*}

In particular, we show the following improvement of the first author's $A_2$ theorem for Calder\'on--Zygmund operators $T$:
\begin{equation*}
  \Norm{T}{\mathscr{B}(L^2(w))}\leq c_T\,[w]_{A_2}^{1/2}\big([w]_{A_\infty}'+[w^{-1}]_{A_\infty}'\big)^{1/2}.
\end{equation*}
Corresponding $A_{p}$ type results are obtained from a new extrapolation theorem with appropriate mixed $A_p$--$A_\infty$ bounds. This uses new two-weight estimates for the maximal function, which improve on Buckley's classical bound.  

We also derive mixed $A_1$--$A_\infty$ type results  of Lerner, Ombrosi and the second author (Math. Res. Lett. 2009) of the form:     
\begin{equation*}
\begin{split}
\Norm{T}{\mathscr{B}(L^p(w))}
&\le c\,pp'\, [w]_{A_1}^{1/p}([w]_{A_{\infty}}')^{1/p'}, \qquad  1<p<\infty, \\
 \|Tf\|_{L^{1,\infty}(w)} &\le c[w]_{A_1} \log(e+[w]'_{A_{\infty}})\|f\|_{L^1(w)}.
\end{split}
\end{equation*}
An estimate dual to the last one is also found, as well as new bounds for commutators of singular integrals.
\end{abstract}



\section{Introduction and statements of the main results}

The weights $w$ for which the usual operators $T$ of Classical Analysis (like the Hardy--Littlewood maximal operator, the Hilbert transform, and general classes of Calder\'on--Zygmund operators) act boundedly on $L^p(w)$ were identified in the 1970's in the works of Muckenhoupt, Hunt, Wheeden, Coifman and Fefferman \cite{CF,HMW,Muckenhoupt:Ap}. This class consists of the Muckenhoupt $A_p$ weights, defined by the condition that (see \cite{GCRdF})
\begin{equation*}
  [w]_{A_p}:= 
 \sup_Q\Big(\avgint_Q w\Big)\Big(\avgint_Q w^{-1/(p-1)}\Big)^{p-1}
  <\infty,\qquad p\in(1,\infty),
\end{equation*}
where the supremum is over all cubes in $\R^d$.
Hence it is shown for any of these important operators $T$, whether it is linear or not, that
\begin{equation*}
   \Norm{T}{\mathscr{B}(L^p(w))}:=\sup_{f\neq 0}\frac{\Norm{Tf}{L^p(w)}}{\Norm{f}{L^p(w)}}
\end{equation*}
is finite if and only if $[w]_{A_p}<\infty$.

\

\subsection{The $A_p$ theory} 

\

It is a natural question to look for optimal quantitative bounds of $\Norm{T}{\mathscr{B}(L^p(w))}$ in terms of $[w]_{A_p}$. The first author who studied that question was S. Buckley in his 1992 Ph.D. thesis (see~\cite{Buckley})  who proved 
\begin{equation}\label{pet}
\Norm{M }{\mathscr{B}(L^p(w))}\le
c_{p,d}\,[w]_{A_p}^{\frac{1}{p-1}}
\qquad1<p<\infty,
\end{equation}
where $M$ is the usual Hardy-Littlewood maximal function on ${\mathbb R}^d$. However,  there has been a big impetus on finding such precise dependence for more singular operators after the work of  Astala, Iwaniec and Saksman~\cite{AIS} due to the connections with sharp regularity results for solutions to the Beltrami equation. The key fact was to prove that the operator norm of the Beurling-Ahlfors transform on $L^2(w)$ grows linearly in terms of the $A_2$ constant of $w$. This was
proved by S. Petermichl and A. Volberg~\cite{PV} and by
Petermichl~\cite{Petermichl:Hilbert, Petermichl:Riesz} for the Hilbert transform and
the Riesz transforms. To be precise, in these papers it has been shown that if $T$
is any of these operators, then
\begin{equation}\label{Aptheorem}
\Norm{T}{\mathscr{B}(L^p(w))}\leq c_{p,T}\,[w]_{A_p}^{\max\{1,\frac{1}{p-1}\}}.
\end{equation}
The exponents are optimal in the sense that the exponent cannot be replaced by any smaller quantity.
 It was conjectured then that the same estimate
holds for any Calder\'on-Zygmund operator $T$. This was first proven for special classes of integral transforms \cite{CMP-ERAMS,LPR} and finally, for  general Calder\'on--Zygmund operators by the first author in \cite{Hytonen:A2}, using the main result from \cite{PTV} where it is shown that a weak type estimate is enough to prove the strong type. A direct proof of this result can be found in  \cite{HPTV}.   Other related work are \cite{CMP, HLRSUV,LMPT,Lerner:Ap,Vagharshakyan}. 
 
The main purpose of this paper is to show that these results can be further improved. To do this, we recall the following definitions of the $A_\infty$ constant of a weight $w$:
First, there is the notion introduced by Hru\v{s}\v{c}ev \cite{Hruscev} (see also \cite{GCRdF}):
\begin{equation*}
 [w]_{A_\infty}:=\sup_Q\Big(\avgint_Q w\Big)\exp\Big(\avgint_Q \log w^{-1}\Big),
\end{equation*}
and second, the smaller (as it turns out) quantity, which appeared with a different notation in the work of Wilson 
\cite{Wilson:87,Wilson:89,  Wilson-LNM} and was recently termed the `$A_\infty$ constant' by Lerner \cite[Section~5.5]{Lerner:Ap}:
\begin{equation*}
  [w]_{A_\infty}':=\sup_Q\frac{1}{w(Q)}\int_Q M(w\chi_Q).
\end{equation*}
Observe that 
\begin{equation*}
  c_d[w]_{A_\infty}'\leq [w]_{A_\infty}\leq[w]_{A_p},\qquad\forall p\in[1,\infty),
\end{equation*}
where the second estimate is elementary, and the first will be checked in Proposition~\ref{prop:2Ainftys}. While the constant $[w]_{A_\infty}$ is more widely used in the literature, and also more flexible for our purposes, it is of interest to observe situations, where the smaller constant $[w]_{A_\infty}'$ is sufficient for our estimates, thereby giving a sharper bound.

Now, if $\si=w^{-1/(p-1)}$ is  the dual weight of $w$, we also have $[\si]_{A_\infty}^{p-1}\leq[\si]_{A_{p'}}^{p-1}=[w]_{A_p}$. The point here is that  these quantities can be much smaller for some classes of weights. Our results will be of the form
\begin{equation*}
   \Norm{T}{\mathscr{B}(L^p(w))}\leq c_{p,T}\,\sum [w]_{A_p}^{\alpha(p)}[w]_{A_\infty}^{\beta(p)}[\si]_{A_\infty}^{(p-1)\gamma(p)},
\end{equation*}
sometimes even with the smaller $[\ ]_{A_\infty}'$ constant instead of $[\ ]_{A_\infty}$, 
where the sum is over at most two triplets $(\alpha,\beta,\gamma)$, and the exponents satisfy $\alpha(p)+\beta(p)+\gamma(p)=\tau(p)$, where $\tau(p)$ is the exponent from the earlier sharp results, as it should. However, we will have $\alpha(p)<\tau(p)$, which shows that part of the necessary $A_p$ control may in fact be replaced by weaker $A_\infty$ control.

As in the usual case, a key point to understand our main result for Calder\'on--Zygmund operators is to consider first the case $p=2$.  

\begin{theorem}\label{thm:theA2&Ainfty}
Let $T$ be a Calder\'on--Zygmund operator, $w\in A_2$ and $\si=w^{-1}$. Then there is a constant $c=c_{d,T}$ such that 
\begin{equation}\label{eq:theA2&Ainfty}
\begin{split}
 \Norm{T}{\mathscr{B}(L^2(w))}
 &\leq c[w]_{A_2}^{1/2}\big( [w]_{A_\infty}'+[w^{-1}]_{A_\infty}' \big)^{1/2} \\
 &\leq c[w]_{A_2}^{1/2}\big([w]_{A_\infty}+[w^{-1}]_{A_\infty}\big)^{1/2}.
\end{split}
\end{equation}
\end{theorem}

We will prove this by following the approach form \cite{Hytonen:A2,HPTV} to the $A_2$ theorem $\Norm{T}{\mathscr{B}(L^2(w))}\leq c_{T}\,[w]_{A_2}$, and modifying the proof at some critical points. Indeed, the original argument uses the $A_2$ property basically \emph{twice}, each producing the factor $[w]_{A_2}^{1/2}$, and it suffices to observe that only the $A_\infty$ property is actually needed in one of these estimates. 

An interesting consequence of this theorem is the following:  for any fixed Calder\'on--Zygmund operator $T$, we have
\begin{equation}\label{eq:reverseA2false}
  \inf_{w\in A_2}\frac{\Norm{T}{\mathscr{B}(L^2(w))}}{[w]_{A_2}}=0.
\end{equation}
This follows once we describe, in Section~\ref{examples}, a family of weights $w\in A_2$ for which both $[w]_{A_\infty}'$ and $[\si]_{A_\infty}'$ (and even $[w]_{A_\infty}$ and $[\si]_{A_\infty}$) grow slower than $[w]_{A_2}$.
In particular, the ``reverse $A_2$ conjecture'' $[w]_{A_2}\leq c_T\,\Norm{T}{\mathscr{B}(L^2(w))}$ is false.

The result for $p\neq 2$ will be  obtained by means of a new quantitative variant of the Rubio de Francia extrapolation theorem adapted to the $A_\infty$ control, which we discuss further below.

\begin{corollary}\label{cor:theAp&Ainfty}
Let $T$ be a Calder\'on--Zygmund operator and let  $p\in(1,\infty)$.  Then if $w\in A_p$ and $\si=w^{-1/(p-1)}$, we have
\begin{equation}\label{cor:CZlower}
\begin{split}
  \Norm{T}{\mathscr{B}(L^p(w))}
  &\lesssim[w]_{A_p}^{2/p-1/2}\Big([w]_{A_\infty}^{1/2}+[\si]_{A_\infty}^{(p-1)/2}\Big)([\si]_{A_\infty}')^{2/p-1}  \\
   &\lesssim[w]_{A_p}^{2/p}([\si]_{A_\infty}')^{2/p-1}\qquad\text{for }p\in(1,2], \\
\end{split}
\end{equation}
and
\begin{equation}\label{cor:CZupper}
\begin{split}
  \Norm{T}{\mathscr{B}(L^p(w))}
  &\lesssim[w]_{A_p}^{2/p-1/[2(p-1)]}\Big([w]_{A_\infty}^{1/[2(p-1)]}+[\si]_{A_\infty}^{1/2}\Big)([w]_{A_\infty}')^{1-2/p} \\
  &\lesssim[w]_{A_p}^{2/p}([w]_{A_\infty}')^{1-2/p}\qquad\text{for }p\in[2,\infty).
\end{split}
\end{equation}
\end{corollary}

Here the simpler forms of the estimates in \eqref{cor:CZlower} and \eqref{cor:CZupper} are almost as good as the more complicated ones, since for many common weights, like power weights, we have $[w]_{A_\infty}+[\si]_{A_\infty}^{p-1}\eqsim[w]_{A_p}$; see Section~\ref{examples}.

These two statements \eqref{cor:CZlower} and \eqref{cor:CZupper} are actually equivalent to each other by using
\begin{equation*}
  \Norm{T}{\mathscr{B}(L^p(w))}=\Norm{T^*}{\mathscr{B}(L^{p'}(\si))}
\end{equation*}
and the fact that $T^*$ is also a Calder\'on--Zygmund operator.

\

\subsection{The maximal function and extrapolation with $A_\infty$ control} 

\

As mentioned above, a key point to understand Corollary \ref{cor:theAp&Ainfty}  is to prove a version  of the quantitative extrapolation theorem adapted to $A_\infty$ control. The proof of this theorem requires to study the corresponding question for the Hardy--Littlewood maximal function, which we first do in a two-weight setting. We need 
a new two-weight constant $B_p[w,\si]$  defined by the functional
\begin{equation}
B_p[w,\si]:=\sup_Q\Big(\avgint_Q w\Big)\Big(\avgint_Q \si\Big)^{p}\exp\Big(\avgint_Q\log\si^{-1}\Big).
\end{equation}
which clearly satisfies
\begin{equation*}
 [w]_{A_p}\leq B_p[w,\si]\leq [w]_{A_p}[\si]_{A_\infty}.
\end{equation*}

\begin{theorem}\label{sharpBuckley} Let $M$ be the Hardy-Littlewood  maximal operator and let $p\in (1,\infty)$. Then we have the estimates
\begin{equation}\label{eq:BuckleyBp}
  \Norm{M (f\si)}{L^p(w)}
  \leq C_d\cdot p'\cdot \big(B_p[w,\si]\big)^{1/p}\Norm{f}{L^p(\si)},
\end{equation}
and
\begin{equation}\label{eq:BuckleyApAinfty}
  \Norm{M (f\si)}{L^p(w)}
  \leq C_d\cdot p'\cdot \big([w]_{A_p}[\si]_{A_\infty}'\big)^{1/p}\Norm{f}{L^p(\si)}.
\end{equation}

\end{theorem}

We refer to Section \ref{two weight section} for the proof and for more information and background about this two-weight estimate for $M$.  By a well-known change-of-weight argument, \eqref{eq:BuckleyApAinfty} implies:

\begin{corollary}\label{CorsharpBuckley}   For $M$ and $p$ as above, and $\si=w^{-1/(p-1)}$, we have
\begin{equation}\label{HLLpAinfty}
  \Norm{M}{\mathscr{B}(L^p(w))}\leq C_d\cdot p'\cdot \big([w]_{A_p}[\si]_{A_\infty}'\big)^{1/p}.
\end{equation}
\end{corollary}

This improves on Buckley's theorem $\Norm{M}{\mathscr{B}(L^p(w))}\leq C_d\cdot p'\cdot [w]_{A_p}^{1/(p-1)}$. Corollary~\ref{CorsharpBuckley}, at least for $p=2$, was also independently discovered by A.~Lerner and S.~Ombrosi \cite{LO:personal}.

\

We now recall the following quantitative version of Rubio de Francia's classical extrapolation theorem due to Dragi\v{c}evi\'c, Grafakos, Pereyra, and Petermichl \cite{DGPP}: If an operator $T$ satisfies
\begin{equation*}
  \Norm{T}{\mathscr{B}(L^r(w))}
  \leq\varphi([w]_{A_r})
\end{equation*}
for a fixed increasing function $\varphi$ and for all $w\in A_r$, then it satisfies a similar estimate for all $p\in(1,\infty)$:
\begin{equation*}
  \Norm{T}{\mathscr{B}(L^p(w))}
  \leq 2\varphi\left(c_{p,r,d}[w]_{A_p}^{\max\{1,(r-1)/(p-1)\}}\right);
\end{equation*}
in particular, $\Norm{T}{\mathscr{B}(L^r(w))}\lesssim[w]_{A_r}^{\tau(r)}$ implies that
\begin{equation*}
  \Norm{T}{\mathscr{B}(L^p(w))}\lesssim[w]_{A_p}^{\tau(r)\max\{1,(r-1)/(p-1)\}}.
\end{equation*}

With our new quantitative estimates involving both $A_p$ and $A_\infty$ control, it seems of interest to extrapolate such bounds as well. Hence we consider weighted estimates of the form
\begin{equation}\label{eq:generalBound}
  \Norm{Tf}{L^r(w)}
  \leq\varphi\left([w]_{A_r}, [w]_{A_{\infty}},[w^{-1/(r-1)}]_{A_{\infty}}^{(r-1)}\right)   \Norm{f}{L^r(w)}
\end{equation}
where $\varphi:[1,\infty)^3\to[0,\infty)$ is an increasing function with respect to each of the variables.
An example is our bound for singular integrals \eqref{thm:theA2&Ainfty}, where 
\begin{equation}\label{eq:examplePhi}
  \varphi(x,y,z)=Cx^{1/2}(y+z)^{1/2}.
\end{equation}
We now aim to extrapolate bounds like \eqref{eq:generalBound} from the given $r\in(1,\infty)$ to other exponents $p\in(1,\infty)$. 

\begin{theorem}[Lower Extrapolation] \label{thm:lowerExtrap}
Suppose that for some $r$ and every $w\in A_r$, an operator $T$ satisfies \eqref{eq:generalBound}.
Then for every $p\in(1,r)$, it satisfies
\begin{equation*}
\begin{split}
  \Norm{Tf}{L^p(w)}
  &\leq 2\varphi\left((2\Norm{M}{\mathscr{B}(L^p(w))})^{r-p}\big([w]_{A_r}, [w]_{A_{\infty}},[w^{-1/(p-1)}]_{A_{\infty}}^{(p-1)}\big)\right)   \Norm{f}{L^p(w)} \\
  &\leq 2\varphi\left((c_d([w]_{A_p}[w^{-1/(p-1)}]_{A_\infty}')^{1/p})^{r-p}\right. \\
    &\qquad\qquad\times \left.\big([w]_{A_p}, [w]_{A_{\infty}},[w^{-1/(p-1)}]_{A_{\infty}}^{(p-1)}\big)\right)   \Norm{f}{L^p(w)}
\end{split}
\end{equation*}
\end{theorem}

In typical applications, like \eqref{eq:examplePhi}, the function $\varphi$ will have a homogeneity of the form $\varphi(\lambda x,\lambda y,\lambda z)=\lambda^s\varphi(x,y,z)$, and hence the common factor
\begin{equation*}
  (2\Norm{M}{\mathscr{B}(L^p(w))})^{r-p}\leq (c_d([w]_{A_p}[w^{-1/(p-1)}]_{A_\infty}')^{1/p})^{r-p}
\end{equation*}
may be extracted out of $\varphi$.

Observe that the condition \eqref{eq:generalBound} is of course implied by the stronger inequality
\begin{equation*}
  \Norm{Tf}{L^r(w)}
  \leq\varphi\big([w]_{A_r},c_d^{-1}[w]_{A_{\infty}}',(c_d^{-1}[w^{-1/(r-1)}]_{A_{\infty}}')^{(r-1)}\big)\Norm{f}{L^r(w)};
\end{equation*}
however, even if we have this stronger inequality to start with (as is the case with the $A_2$ theorem for Calder\'on--Zygmund operators), we do not know how to exploit it to get a stronger conclusion than what we can derive from \eqref{eq:generalBound}. A related difficulty will be pointed out in the proof. This is why we restrict to the assumption \eqref{eq:generalBound} only.

\begin{theorem}[Upper Extrapolation]  \label{thm:upperExtrap}
Suppose that for some $r$ and every $w\in A_r$, an operator $T$ satisfies \eqref{eq:generalBound}.
Then for every $p\in(r,\infty)$, it satisfies
\begin{equation*}
\begin{split}
 \Norm{Tf}{L^p(w)}
  &\leq 2\varphi\left((2\Norm{M}{\mathscr{B}(L^{p'}(w^{1-p'}))})^{(p-r)/(p-1)}\right.\\
    &\qquad\qquad\times\left.\big([w]_{A_p}^{(r-1)/(p-1)},[w]_{A_\infty}^{(r-1)/(p-1)},
         [w^{-1/(p-1)}]_{A_\infty}^{(r-1)}\big)\right)\Norm{f}{L^p(w)} \\
  &\leq 2\varphi\left(\big(c_d[w]_{A_p}^{1/p}([w]_{A_\infty}')^{1/p'}\big)^{(p-r)/(p-1)}\right.\\
    &\qquad\qquad\times\left.\big([w]_{A_p}^{(r-1)/(p-1)},[w]_{A_\infty}^{(r-1)/(p-1)},
         [w^{-1/(p-1)}]_{A_\infty}^{(r-1)}\big)\right)\Norm{f}{L^p(w)}.
\end{split}
\end{equation*}
\end{theorem}

It is immediate to check that Theorems~\ref{thm:lowerExtrap} and \ref{thm:upperExtrap}, in combination with Theorem~\ref{thm:theA2&Ainfty}, give Corollary~\ref{cor:theAp&Ainfty}. In fact, thanks to the mentioned equivalence of the two parts \eqref{cor:CZlower} and \eqref{cor:CZupper} of  Corollary~\ref{cor:theAp&Ainfty}, we would only need one of Theorems~\ref{thm:lowerExtrap} and \ref{thm:upperExtrap} to deduce this corollary.
But for other classes of operators without a self-dual structure, it is useful to have both upper and lower extrapolation results available.

\

\subsection{The $A_1$ theory} 

\

It is an interesting fact that if we assume that the weight satisfy the stronger condition $w\in A_1$, then the estimate \eqref{Aptheorem} can be considerably improved. Indeed, if $T$ is any  Calder\'on--Zygmund operator, then $T$ is of course  bounded on $L^p(w)$, because $A_1\subset A_p$  but with a much better bound, namely
\begin{equation} \label{A1strong(p,p)}
\|T\|_{\mathscr{B}(L^p(w))}\le c\,pp'\, [w]_{A_1},\qquad  1<p<\infty.
\end{equation}
Observe that the dependence on the $A_1$ constant is linear for any $p$ while in the $A_p$ case it is highly nonlinear for $1<p<2$, see \eqref{Aptheorem}.  The result is sharp both in terms of the dependence on $[w]_{A_1}$, and in terms of the dependence on $p$ when taking $w=1$ by the classical theory. This fact was used to get the following endpoint result: 
\begin{equation}\label{LLog}
\|Tf\|_{L^{1,\infty}(w)}\le c[w]_{A_1}\log(e +[w]_{A_{1}})\|f\|_{L^1(w)}.
\end{equation}
See  \cite{LOP3} and also \cite{LOP1} for these results and for more information about the problem.  It was conjectured in \cite{LOP3} that 
the growth of this bound would be linear; however, it  has been recently shown in \cite{NRVV}  that the growth of the bound is worse than linear. It seems  that most probably the $L \log L$  result \eqref{LLog} is the best possible. 

On the other hand, in \cite{LOP2} a sort of ``dual'' estimate to the last bound was found, which is also of interest for related matters:
$$
\Big\| \frac{Tf}{w} \Big\|_{L^{1,\infty}(w)} \le c\, [w]_{A_{1}}\log(e+[w]_{A_1})\,\int_{{\mathbb R}^d}|f|\, \ud x.
$$

In this paper we improve these results following our new quantitative estimates involving this time $A_1$ and $A_\infty$ control. To be precise, we will prove the following new results: 

\begin{theorem}\label{thm:A1strong(p,p)}   Let $T$ be a Calder\'on-Zygmund operator and let $1<p<\infty$. Then
\begin{equation*}
\|T\|_{\mathscr{B}(L^p(w))}\le c\,pp'\, [w]_{A_1}^{1/p}  ([w]_{A_\infty}')^{1/p'}
\end{equation*}
where $c=c(d,T)$.
\end{theorem}

We will prove this by following the approach from  \cite{LOP1, LOP3} to \eqref{A1strong(p,p)}, modifying the proof at several points.
In analogy to \eqref{eq:reverseA2false}, Theorem~\ref{thm:A1strong(p,p)} disproves the ``reverse $A_1$ conjecture'' $[w]_{A_1}\leq c_T\,\Norm{T}{\mathscr{B}(L^p(w))}$ for all $p\in (1,\infty)$: considering a family of weights $w\in A_1$ for which  $[w]_{A_\infty}$ grow slower than $[w]_{A_1}$, for any fixed Calder\'on--Zygmund operator $T$, we have (see Section~\ref{examples} for details)
\begin{equation*}
\inf_{w\in A_1}\frac{\Norm{T}{\mathscr{B}(L^p(w))}}{[w]_{A_1}}=0, \quad 1<p<\infty.
\end{equation*}

Finally we will also use the approach from  \cite{LOP3} and  \cite{LOP2} to prove the following theorems respectively.

\begin{theorem}  \label{thm:A1weak(1,1)} Let $T$ be a Calder\'on-Zygmund operator.
Then
\begin{equation*}
\|Tf\|_{L^{1,\infty}(w)}\le c_{d,T}\,[w]_{A_1}\log (e+  [w]_{A_\infty}'  )\|f\|_{L^1(w)}.
\end{equation*}
\end{theorem}

\begin{theorem}  \label{thm:dualA1weak(1,1)} Let $T$ be a Calder\'on-Zygmund operator.
Then
\begin{equation*}
\Big\| \frac{Tf}{w} \Big\|_{L^{1,\infty}(w)} \le c_{d,T}\, [w]_{A_\infty}' \, \log(e+[w]_{A_1}))\,  \|f\|_{L^{1}({\mathbb R}^d)}.
\end{equation*}
\end{theorem}

\

\subsection{Commutators with BMO functions} \label{commutators}

\

We further pursue the $A_\infty$  point-of-view by proving  a result  in the spirit of  Theorem \ref{thm:theA2&Ainfty} for commutators of linear operators $T$ with BMO functions.  These operators are defined formally by
the expression
$$[b,T]f=b T(f) - T(b\,f). $$
More generally we can consider the  $k$th order commutator defined by 
$$T_b^k:= [b, T^{k-1}_b].$$

When $T$ is a singular integral operator, these operators were
considered by Coifman, Rochberg and Weiss \cite{CRW-commutators} and since then many results have been obtaind. We refer to \cite{CPP} for more information about these operators. It is shown in  \cite{CPP} that if $T$ is a linear operator bounded on $L^2(w)$ for any $w\in A_2$ with bound 
\begin{equation*}
\|T\|_{\mathscr{B}(L^{2}(w))}  \le \vp([w]_{A_2}),
\end{equation*}
where $\vp$ is an increasing function  $\vp:[1,\infty)\to
[0,\infty)$, then there is a dimensional constant  $c$ such that
\begin{equation*}
\|[b,T]\|_{\mathscr{B}(L^{2}(w))} \le c\, \vp(c[w]_{A_2})\,[w]_{A_2}
\|b\|_{BMO}.
\end{equation*}
In particular, if $T$ is any  Calder\'on-Zygmund operator we can use the linear $A_2$ theorem for $T$ to deduce
\begin{equation*}
\|[b,T]\|_{\mathscr{B}(L^{2}(w))} \le c\,[w]_{A_2}^2\,
\|b\|_{BMO},
\end{equation*}
and the quadratic exponent cannot be improved. 

An analogous result adapted to the $A_\infty$ control reads as follows:

\begin{theorem}\label{main}
Let $T$ be a linear operator bounded on $L^2(w)$ for any $w\in A_2$ and let $b\in BMO$.
Suppose further that there is a function  $\vp:[1,\infty)^3\to
[0,\infty)$,  increasing with respect to each component, such that
\begin{equation*}
\|T\|_{\mathscr{B}(L^{2}(w))} \le \vp\Big(  [w]_{A_2},[w]_{A_\infty}',[\si]_{A_{\infty}}'\Big).
\end{equation*}
then there is a dimensional constant $c$ such that
\begin{equation*}
\|[b,T]\|_{\mathscr{B}(L^{2}(w))} \le c\, \vp\Big(c\,[w]_{A_2},c\,[w]_{A_\infty}',c\,[\si]_{A_{\infty}}'\Big)\, 
\big([w]_{A_\infty}'+[\si]_{A_{\infty}}'\big)\,
\|b\|_{BMO},
\end{equation*}
or more generally 
\begin{equation*}
\|T_b^k\|_{\mathscr{B}(L^{2}(w))} \le c\, \vp\Big(c\,[w]_{A_2},c\,[w]_{A_\infty}',c\,[\si]_{A_{\infty}}'\Big)\, 
\big([w]_{A_\infty}'+[\si]_{A_{\infty}}'\big)^k\,
\|b\|^k_{BMO}.
\end{equation*}
\end{theorem}

We can now apply Theorem \ref{thm:theA2&Ainfty}.

\begin{corollary}\label{corollary:commutatorCZO}
Let $T$ be  any Calder\'on-Zygmund operator and let $b\in BMO$. Then 
\begin{equation*}
\|[b,T]\|_{\mathscr{B}(L^{2}(w))} \leq c\,[w]_{A_2}^{1/2}\big([w]_{A_\infty}'+[w^{-1}]_{A_\infty}'\big)^{3/2}\,\|b\|_{BMO},
\end{equation*}
or more generally, 
\begin{equation*}
\|T_b^k\|_{\mathscr{B}(L^{2}(w))} \le c\,  [w]_{A_2}^{1/2}\, 
\big([w]_{A_\infty}'+[\si]_{A_{\infty}}'\big)^{k+1/2}\,
\|b\|^k_{BMO}.
\end{equation*}
\end{corollary}

\

\subsection{An end-point estimate when $p=\infty$}

\

Having investigated the sharp bounds for Calder\'on-Zygmund operators $T:L^p(w)\to L^p(w)$ for $p\in(1,\infty)$ and $w\in A_p$, we finally consider the end-point
\begin{equation*}
  T:L^{\infty}(w)\to\BMO(w),\qquad w\in A_\infty.
\end{equation*}
Qualitatively, this situation seems slightly uninteresting, as these end-point spaces simply reduce to their unweighted analogues: That $L^{\infty}(w)=L^{\infty}$ with equal norms is immediate from the fact that $w$ and the Lebesgue measure share the same zero sets for $w\in A_{\infty}$. That the weighted norm
\begin{equation*}
  \Norm{f}{\BMO(w)}:=\sup_Q\inf_c\frac{1}{w(Q)}\int_Q\abs{f-c}w<\infty,
\end{equation*}
is equivalent to the usual $\Norm{f}{\BMO}$ for $w\in A_\infty$ was proven by Muckenhoupt and Wheeden \cite{MW}, Theorem~5. However, one may still investigate the quantitative bound of operators $T:L^{\infty}\to\BMO=\BMO(w)$, when the latter space is equipped with the norm $\Norm{\ }{\BMO(w)}$. We start with:

\begin{theorem}\label{thm:embNorm}
For $w\in A_{\infty}$, we have a bounded embedding $I:\BMO\hookrightarrow\BMO(w)$ of norm at most $c[w]_{A_{\infty}}'$, where $c$ is dimensional. This estimate is sharp in the following sense: if the norm of the embedding is bounded by $\phi([w]_{A_{\infty}}')$, or just by $\phi([w]_{A_\infty})$, for all $w\in A_{\infty}$, then $\phi(t)\geq ct$.
\end{theorem}

The following corollary for Calder\'on--Zygmund operators can be seen as an easy endpoint estimate of the bound $\Norm{T}{\mathscr{B}(L^p(w))}\leq c_{p,T}\,[w]_{A_p}$ for $p\in[2,\infty)$.

\begin{corollary}\label{cor:LinftyBMO}
Let $T$ be  any Calder\'on-Zygmund operator and let $w\in A_{\infty}$, then \, $T:L^{\infty}\to\BMO(w)$ with norm at most $c_{T}\,[w]_{A_{\infty}}'$.  Furthermore,  this estimate is sharp in terms of the dependence on $[w]_{A_{\infty}}'$ in the same way as Theorem~\ref{thm:embNorm}.
\end{corollary}

We conclude the introduction by stating the following observation which may be of some interest.

\begin{proposition}\label{BMO-Ainfty}
If $w\in A_\infty$, then $\log w\in\BMO$ with 
\begin{equation*}
  \Norm{\log w}{\BMO}\leq\log(2e[w]_{A_\infty}).
\end{equation*}
\end{proposition}

\section{The two different $A_\infty$ constants}

Before pursuing further our analysis of inequalities with $A_\infty$ control, we include this short section to compare the two $A_\infty$ constants
\begin{equation*}
  [w]_{A_\infty}:=\sup_Q\Big(\avgint_Q w\Big)\exp\Big(\avgint_Q\log w^{-1}\Big),\qquad
  [w]_{A_\infty}':=\sup_Q\frac{1}{w(Q)}\int_Q M(w\chi_Q).
\end{equation*}

We need the following auxiliary estimate, which is also used later in the paper:

\begin{lemma}\label{lem:logMaxFn}
The logarithmic maximal function
\begin{equation*}
  M_0 f:=\sup_{Q}\exp\Big(\avgint_Q\log\abs{f}\Big) \,\chi_Q
\end{equation*}
satisfies
\begin{equation*}
  \Norm{M_0f}{L^p}\leq c_{d}^{1/p}\Norm{f}{L^p}
\end{equation*}
for all $p\in(0,\infty)$. For the dyadic version, we can take $c_{d}=e$, independent of dimension $d$.
\end{lemma}

\begin{proof}
By Jensen's inequality and the basic properties of the logarithm, we have
\begin{equation*}
  M_0 f\leq Mf,\qquad M_0 f=(M_0\abs{f}^{1/q})^q\leq(M\abs{f}^{1/q})^q,\qquad q\in(0,\infty),
\end{equation*}
where $M$ is the Hardy--Littlewood maximal operator, or the dyadic maximal operator in the case of dyadic $M_0$.
By the $L^q$ boundedness of the usual maximal function for $q>1$, we have
\begin{equation*}
  \int[M_0 f]^p\leq\int[M\abs{f}^{p/q}]^q
  \leq(C_d\cdot q')^q\int(\abs{f}^{p/q})^q=(C_d\cdot q')^q\int\abs{f}^p.
\end{equation*}
In the non-dyadic case, we simply take, say, $q=2$, giving the claim with $c_d=(2C_d)^2$. In the dyadic case, we have $C_d=1$, and we can take the limit $q\to\infty$, which gives
\begin{equation*}
  (q')^q=\Big(\frac{q}{q-1}\Big)^q=\Big(1+\frac{1}{q-1}\Big)^q\to e,
\end{equation*}
and hence $\Norm{M_0 f}{L^p}^p\leq e\Norm{f}{L^p}^p$.
\end{proof}

\begin{proposition}\label{prop:2Ainftys}
We have $[w]_{A_\infty}'\leq c_d[w]_{A_\infty}$, where $c_d$ is as in Lemma~\ref{lem:logMaxFn}.
\end{proposition}

\begin{proof}
For $x\in Q$, it is not difficult to see that for the computation of $M(w\chi_Q)(x)$, it suffices to take the supremum over cubes $R\owns x$ with $R\subseteq Q$:
\begin{equation*}
  M(w\chi_Q)(x)=\sup_{\substack{R\owns x\\ R\subseteq Q}}\avgint_R w,\qquad\forall x\in Q.
\end{equation*}
By definition of $[w]_{A_\infty}$, we have
\begin{equation*}
  \avgint_R w\leq [w]_{A_\infty}\exp\Big(\avgint_R\log w\Big),
\end{equation*}
and hence, taking the supremum over $R$,
\begin{equation*}
  M(w\chi_Q)(x)\leq [w]_{A_\infty} M_0(w\chi_Q)(x),\qquad\forall x\in Q.
\end{equation*}
Integration over $Q$ and application of Lemma~\ref{lem:logMaxFn} now give
\begin{equation*}
  \int_Q M(w\chi_Q)\leq [w]_{A_\infty}\int M_0(w\chi_Q)\leq [w]_{A_\infty}c_d\int w\chi_Q=c_d[w]_{A_\infty}w(Q);
\end{equation*}
thus $[w]_{A_\infty}'\leq c_d[w]_{A_\infty}$.
\end{proof}

It is a well known fact that any $A_{\infty}$ weight satisfies a reverse  reverse H\"older inequality playing a central role in the area. 
In this paper a sharp version of this property will  also play a fundamental role. To be precise if $w \in A_{\infty}$ we define 
$$r(w):=1+\frac{1}{ \tau_d\,[w]_{A_{\infty}}'}$$ 
where $\tau_d$ is a dimensional constant that we may take to be $\tau_d= 2^{11+d}$.  Note that $r(w)'\approx [w]_{A_{\infty}}'$. The result we need is the following.

\begin{theorem}[A new sharp reverse H\"older inequality]\label{thm:SharpRHI}

\

\vspace{0.2cm}

a)  If $w \in A_{\infty}$, then
\begin{equation*}
  \Big(\avgint_Q w^{r(w)}\Big)^{1/r(w)}\leq 2\avgint_Q w.
\end{equation*}

b) Furthermore, the result is optimal up to a dimensional factor:  If a weight $w$ satisfies the reverse H\"older inequality
\begin{equation*}
  \Big(\avgint_Q w^r\Big)^{1/r}\leq K\avgint_Q w,
\end{equation*}
then $[w]_{A_\infty}'\leq  c_d \cdot K\cdot r'$.
\end{theorem}

This result is new in the literature and has its own interest. 
In the classical situation most  of the available proofs do not give such explicit constants, which are important for us. Only under the stronger condition of $A_1$ was found and used in a crucial way in \cite{LOP3}. Very recently a very nice proof by A.~de la Torre~\cite{dlT} for the case $[w]_{A_{\infty}}$ was sent to us. Another less precise proof and for the $A_p$ case, $1<p<\infty$, can be found in  \cite{perez-lecturenotes}.

Part b) follows from the boundedness of the maximal function in $L^r$ with constant $c_d r'$: 
\begin{equation*}
\begin{split}
   \avgint_Q  M(\chi_Q w)
   &\leq \Big(\avgint_Q M(\chi_Q w)^r\Big)^{1/r} \\
   &\leq  c_d \cdot r'\Big(\avgint_Q w^r\Big)^{1/r}
   \leq c_d \cdot r'\cdot K\avgint_Q w.
\end{split}
\end{equation*}

\section{The $A_2$ theorem for Calder\'on--Zygmund operators}

The purpose of this section is to prove Theorem \ref{thm:theA2&Ainfty}, namely, the estimate
\begin{equation*}
  \Norm{T}{\mathscr{B}(L^2(w))}\leq c[w]_{A_2}^{1/2}\big([w]_{A_\infty}'+[\sigma]_{A_\infty}'\big)^{1/2},
\end{equation*}
where $c=c_{d,T}$ is a constant depending on the dimension and the operator $T$.

Here and throughout this section, $\sigma=w^{-1}$. This improves on the $A_2$ theorem \cite{Hytonen:A2}:
\begin{equation*}
  \Norm{T}{\mathscr{B}(L^2(w))}\leq c\,[w]_{A_2},
\end{equation*}
and its proof follows the same outline, with the implementation of the $A_\infty$ philosophy at certain key points.

\subsection{Reduction to a dyadic version}

\

Fundamental to this proof strategy is the notion of \emph{dyadic shifts}, which we recall. We work with a \emph{general dyadic system} $\mathscr{D}$, this being a collection of axis-parallel cubes $Q$, whose sidelengths $\ell(Q)$ are of the form $2^k$, $k\in\Z$, where moreover $Q\cap R\in\{Q,R,\varnothing\}$ for any two $Q,R\in\mathscr{D}$, and the cubes of a fixed sidelength $2^k$ form a partition of $\R^d$. Given such a dyadic system, a dyadic shift with parameters $(m,n)$ is an operator of the form
\begin{equation*}
  \sha f=\sum_{K\in\mathscr{D}}A_K f,
  \qquad A_K f=\frac{1}{\abs{K}}\sum_{\substack{I,J\in\mathscr{D};I,J\subseteq K \\ \ell(I)=2^{-m}\ell(K) \\ \ell(J)=2^{-n}\ell(K)}}
     \pair{h_I^J}{f}k_J^I,
\end{equation*}
where $h_I^J$ is a generalized Haar function on $I$ (supported on $I$, constant on its dyadic subcubes, and normalized by $\Norm{h_I^J}{\infty}\leq 1$), and $k_J^I$ on $J$. This implies that $\abs{A_K f}\leq \chi_K\cdot\abs{K}^{-1}\cdot\int_K\abs{f}$. For any subcollection $\mathscr{Q}\subset\mathscr{D}$, we write
\begin{equation}\label{eq:subshift}
  \sha_{\mathscr{Q}}f:=\sum_{K\in\mathscr{Q}}A_K f,
\end{equation}
and we require that $\Norm{\sha_{\mathscr{Q}} f}{L^2}\leq\Norm{f}{L^2}$ for all $\mathscr{Q}\subset\mathscr{D}$. This is automatic from straightforward orthogonality considerations in case we only have cancellative Haar functions with $\int h_I^J=\int k_J^I=0$.

Dyadic shifts with parameters $(0,0)$ are well known in dyadic harmonic analysis under different names. Auscher et al.~\cite{AHMTT} study such operators under the name \emph{perfect dyadic} operators, which they decompose into a sum of a \emph{Haar multiplier} (or \emph{martingale transform}), a \emph{paraproduct}, and a \emph{dual paraproduct}. These three types of operators are of course well-known since a long time. The first dyadic shift (and this name) with parameters $(0,1)$ was introduced by Petermichl~\cite{Petermichl:shift}, and the definition in the above generality was given by Lacey, Petermichl and Reguera \cite{LPR}.

The importance of these dyadic shifts for the analysis of Calder\'on--Zygmund operators comes from the following:

\begin{theorem}[Dyadic representation theorem; \cite{Hytonen:A2,HPTV}, Theorems~4.2 and~4.1]
Let $T\in  \mathscr{B}(L^2(\R^d))$ be a Calder\'on--Zygmund operator which satisfies the standard estimates with the H\"older continuity exponent $\alpha\in(0,1]$. Then it has the representation
\begin{equation*}
  \pair{g}{Tf}
  =c_{T,d}\,\Exp_{\mathscr{D}}\sum_{m,n=0}^{\infty}2^{-(m+n)\alpha/2}\pair{g}{\sha^{mn}_{\mathscr{D}}f},
\end{equation*}
valid for all bounded and compactly supported functions $f$ and $g$, where $\sha^{mn}_{\mathscr{D}}$ is a dyadic shift with parameters $(m,n)$ related to the dyadic system $\mathscr{D}$, and $\Exp_{\mathscr{D}}$ is the expectation with respect to a probability measure on the space of all generalized dyadic systems; see \cite{Hytonen:A2} for the details of the construction of this probability space.
\end{theorem}

This result was preceded by several versions restricted to special operators $T$: the Beurling--Ahlfors transform by Dragi{\v{c}}evi{\'c} and Volberg \cite{DV}, the Hilbert transform by  Petermichl \cite{Petermichl:shift}, the Riesz transforms by Petermichl, Treil and Volberg \cite{PTV:Riesz}, and all one-dimensional convolution operators with an odd, smooth kernel by Vagharshakyan \cite{Vagharshakyan}. An immediate consequence of the dyadic representation theorem is that Theorem~\ref{thm:theA2&Ainfty} will be a consequence of the following dyadic version: (Similarly, the special cases of the representation theorem all played a role in proving the $A_2$ theorem for the mentioned particular operators.)

\begin{theorem}\label{thm:shaAinfty}
Let $\sha$ be a dyadic shift with parameters $(m,n)$, and $r=\max\{m,n\}$. For $w\in A_2$ and $\si=w^{-1}$, we have
\begin{equation*}
  \Norm{\sha f}{L^2(w)}\leq C(r+1)^2[w]_{A_2}^{1/2}\big([w]_{A_\infty}'+[\si]_{A_{\infty}}'\big)^{1/2}\Norm{f}{L^2(w)}.
\end{equation*}
\end{theorem}

The weighted norm of the shifts, in turn, is most conveniently deduced with the help of the following characterization of their boundedness in a two-weight situation:

\begin{theorem}[\cite{HPTV}, Theorem~3.4]\label{thm:HPTV}
Let $\sha$ be a dyadic shift with parameters $(m,n)$, and $r=\max\{m,n\}$. If for all $Q\in\mathscr{D}$ and some $B$ there holds
\begin{equation*}
  \Big(\int_Q\abs{\sha(\chi_Q\si)}^2 w\Big)^{1/2}\leq B\si(Q)^{1/2},\qquad
  \Big(\int_Q\abs{\sha^*(\chi_Q w)}^2\si\Big)^{1/2} \leq B w(Q)^{1/2},
\end{equation*}
then for a dimensional constant $c$, we have
\begin{equation*}
  \Norm{\sha(f\si)}{L^2(w)}\leq c\,\big((r+1)B+(r+1)^2(A_2[w,\si])^{1/2}\big)\Norm{f}{L^2(\si)},
\end{equation*}
where  $A_2[w,\si]$ is defined by the functional
\begin{equation*}
A_2[w,\si]:=\sup_Q\Big(\avgint_Q w\Big)\Big(\avgint_Q \si\Big).
\end{equation*}
\end{theorem}

Observing that for $\si=w^{-1}$, the last bound is equivalent to
\begin{equation*}
  \Norm{\sha f}{L^2(w)}\leq c\,\big((r+1)B+(r+1)^2 [w]_{A_2}^{1/2}\big)\Norm{f}{L^2(w)},
\end{equation*}
and $[w]_{A_\infty},[\sigma]_{A_\infty}\geq 1$, we are reduced to estimating the quantity $B$ for $\si=w^{-1}$. Since $\sha$ and $\sha^*$ are operators of the same form, and by the symmetry of $w$ and $\sigma$, Theorem~\ref{thm:HPTV} shows that proving Theorem~\ref{thm:shaAinfty} amounts to showing that
\begin{equation*}
  \Big(\int_Q\abs{\sha(w\chi_Q)}^2\si\Big)^{1/2}\leq c\,(r+1)\big([w]_{A_2}[w]_{A_\infty}' w(Q)\big)^{1/2}.
\end{equation*}
We observe that
\begin{equation*}
  \sha(w\chi_Q)
  =\sum_{K\subseteq Q}A_K(w\chi_Q)+\sum_{K\supset Q}A_K(w\chi_Q),
\end{equation*}
and it suffices to consider the two parts separately. The big cubes are immediately handled by the maximal function estimate (see Corollary \ref{CorsharpBuckley}):
\begin{equation}\label{eq:bigCubes}
\begin{split}
   \int_Q&\Babs{\sum_{K\supset Q}A_K(w\chi_Q)}^2\si
   \leq\int_Q\Big(\sum_{K\supset Q}\frac{w(Q)}{\abs{K}}\chi_K\Big)^2\si \\
   &\lesssim\int_Q M_d(w\chi_Q)^2\si\leq[\si]_{A_2}[w]_{A_\infty}' w(Q)=[w]_{A_2}[w]_{A_\infty}' w(Q).
\end{split}
\end{equation}
Hence, to prove Theorem~\ref{thm:shaAinfty}, we are reduced to showing that
\begin{equation}\label{eq:sha2prove}
  \Big(\int_Q\Babs{\sum_{K\subseteq Q}A_K(w\chi_Q)}^2\si\Big)^{1/2}\leq c\,(r+1)\big([w]_{A_2}[w]_{A_\infty}' w(Q)\big)^{1/2}.
\end{equation}
This is the goal for the rest of this section.

\subsection{Proof of the key estimate~\eqref{eq:sha2prove}}

\

We follow the key steps from \cite{Hytonen:A2,HPTV,LPR}.
The collection $\{K\in\mathscr{D}:K\subseteq Q\}$ is first split into $(r+1)$ subcollections according to the value of $\log_2\ell(K)\mod(r+1)$; we henceforth work with one of these subcollections, which we denote by $\mathscr{K}$. This is the step which introduces the factor $(r+1)$, and we will estimate $\sha_{\mathscr{K}}(w\chi_Q)$ with a bound independent of $r$.

The collection $\mathscr{K}$ is further divided into the sets $\mathscr{K}^a$ of those cubes with
\begin{equation}\label{eq:fixA2}
  2^a<\frac{w(Q)}{\abs{Q}}\frac{\si(Q)}{\abs{Q}}\leq 2^{a+1},
\end{equation}
where $a\leq\log_2[w]_{A_2}$.

Among the cubes $K\in\mathscr{K}^a$, we choose the principal cubes $\mathscr{S}^a=\bigcup_{k=0}^{\infty}\mathscr{S}^a_k$ so that $\mathscr{S}^a_0$ consists of the maximal cubes in $\mathscr{K}^a$, and $\mathscr{S}^a_k$ the maximal cubes $S\in\mathscr{K}^a$ contained in some $S'\in\mathscr{S}^a_{k-1}$ with $\si(S)/\abs{S}>2\si(S')/\abs{S'}$. Then
\begin{equation*}
  \mathscr{K}^a=\bigcup_{S\in\mathscr{S}^a}\mathscr{K}^a(S),\qquad\mathscr{K}^a(S):=
  \{K\in\mathscr{K}^a|K\subseteq S,\not\exists S':K\subseteq S'\subset S\}.
\end{equation*}
It follows that
\begin{equation}\label{eq:shaSplit}
  \sha_{\mathscr{K}}(w\chi_Q)
  =\sum_{a\leq\log_2[w]_{A_2}}\sum_{S\in\mathscr{S}^a}\sha_{\mathscr{K}^a(S)}(w\chi_Q),
\end{equation}
where we use the notation from \eqref{eq:subshift}.

To proceed, we recall the following distributional estimate:

\begin{lemma}[\cite{HPTV}, Eq.~(5.26)]
With notation as above, we have
\begin{equation}\label{eq:expEst}
  \sigma\Big(\abs{\sha_{\mathscr{K}^a(S)}(w\chi_Q)}>t\ave{w}_S\Big)\leq Ce^{-ct}\sigma(S),\qquad\forall S\in\mathscr{S}^a,
\end{equation}
where the constants $C$ and $c$ are at worst dimensional. 
\end{lemma}

This is a powerful estimate which readily leads to norm bounds for \eqref{eq:shaSplit}. The following computation, simplifying the corresponding ones from \cite{Hytonen:A2,HPTV,LPR}, is borrowed from \cite{HLMORSU}: Denoting
\begin{equation*}
  E_j(S):=\{j\leq \abs{\sha_{\mathscr{K}^a(S)}(w\chi_Q)}/\ave{w}_S<j+1\}\subseteq S,
\end{equation*}
we have
\begin{equation*}
\begin{split}
  \Big\|\sum_{S\in\mathscr{S}^a} &\sha_{\mathscr{K}^a(S)}(w\chi_Q)\Big\|_{L^2(\sigma)}
    \leq\sum_{j=0}^{\infty}(j+1)\BNorm{\sum_{S\in\mathscr{S}^a}\ave{w}_S\cdot\chi_{E_j(S)}}{L^2(\sigma)} \\
  &=\sum_{j=0}^{\infty}(j+1)\Big(\int\Big[\sum_{S\in\mathscr{S}^a}\ave{w}_S\cdot\chi_{E_j(S)}(x)\Big]^2\sigma(x)\ud x\Big)^{1/2} \\
  &\overset{(*)}{\leq} C\sum_{j=0}^{\infty}(j+1)\Big(\int\sum_{S\in\mathscr{S}^a}\ave{w}_S^2\cdot\chi_{E_j(S)}(x)\sigma(x)\ud x\Big)^{1/2} \\
  &= C\sum_{j=0}^{\infty}(j+1)\Big(\sum_{S\in\mathscr{S}^a}\ave{w}_S^2\cdot\sigma(E_j(S))\Big)^{1/2} \\
  &\leq C\sum_{j=0}^{\infty}(j+1)\Big(\sum_{S\in\mathscr{S}^a}\ave{w}_S^2\cdot Ce^{-cj}\sigma(S)\Big)^{1/2} \qquad\text{(by \eqref{eq:expEst})}\\
  &\leq C\sum_{j=0}^{\infty}e^{-cj}(j+1)\Big(2^a\sum_{S\in\mathscr{S}^a}w(S)\Big)^{1/2}\qquad\text{(by \eqref{eq:fixA2} for $S\in\mathscr{S}^a\subset\mathscr{K}^a$)} \\
  &\leq C\cdot 2^{a/2}\Big(\sum_{S\in\mathscr{S}^a}w(S)\Big)^{1/2}.
\end{split}
\end{equation*}
In $(*)$ we used the fact that at a fixed $x$, the numbers $\ave{w}_S$ for the principal cubes $S\supset E_j(S)\owns x$ increase at least geometrically, so their $\ell^1$ and $\ell^2$ norms are comparable.

We now come to the crucial point, where we can improve the earlier $A_2$ bounds to $A_\infty$:

\begin{lemma}\label{lem:princAinfty}
For the principal cubes as defined above, we have
\begin{equation*}
  \sum_{S\in\mathscr{S}^a} w(S)\leq 2\cdot[w]_{A_\infty}'\cdot w(Q).  
\end{equation*}
\end{lemma}

\begin{proof}
Let
\begin{equation}\label{eq:ESdef}
  E(S):=S\setminus\bigcup_{S'\subsetneq S}S'.
\end{equation}
The union is the union of its maximal members $S'$, which satisfy
\begin{equation*}
 \abs{S'}=\abs{S'}/w(S')\cdot w(S')\leq\tfrac12\abs{S}/w(S)\cdot w(S'), 
\end{equation*}
hence $\sum\abs{S'}\leq\tfrac12\abs{S}$, and thus
\begin{equation}\label{eq:ESest}
   \abs{E(S)}\geq\tfrac12\abs{S}. 
\end{equation}
Therefore 
\begin{equation*}
\begin{split}
   \sum_{S\in\mathscr{S}^a} w(S)
   &=\sum_{S\in\mathscr{S}^a} \frac{w(S)}{\abs{S}}\abs{S} 
     \leq \sum_{S\in\mathscr{S}^a} \frac{w(S)}{\abs{S}} 2\abs{E(S)} \\
   &\leq 2\sum_{S\in\mathscr{S}^a}\int_{E(S)} M(w\chi_Q) 
     = 2\int_{Q} M(w\chi_Q) \leq 2\cdot [w]_{A_\infty}'\cdot w(Q),
\end{split}
\end{equation*}
where the last step was the definition of $[w]_{A_\infty}'$.
\end{proof}

Substituting the obtained estimates back to \eqref{eq:shaSplit}, we conclude that
\begin{equation*}
\begin{split}
  \Norm{\sha_{\mathscr{K}}(w\chi_Q)}{L^2(\si)}
  &\leq\sum_{a\leq\log_2[w]_{A_2}}\BNorm{\sum_{S\in\mathscr{S}^a}\sha_{\mathscr{K}^a(S)}(w\chi_Q)}{L^2(\sigma)} \\
  &\leq C\sum_{a\leq\log_2[w]_{A_2}}2^{a/2}\Big(\sum_{S\in\mathscr{S}^a}w(S)\Big)^{1/2}\\
  &\leq C\sum_{a\leq\log_2[w]_{A_2}}2^{a/2}\big([w]_{A_\infty}'\cdot w(Q)\big)^{1/2}\\
  &\leq C [w]_{A_2}^{1/2}([w]_{A_{\infty}}')^{1/2}w(Q)^{1/2}.
\end{split}
\end{equation*}
Recalling the initial splitting of $\{K\in\mathscr{D}:K\subseteq Q\}$ into $r+1$ subcollections of the same form as $\mathscr{K}$, this concludes the proof of \eqref{eq:sha2prove}, and hence the proof of Theorem~\ref{thm:shaAinfty}.

\section{Two-weight theory for the maximal function}\label{two weight section}

\subsection{Background}

\

The two-weight problem was  studied in the 1970's by Muckenhoupt and Wheeden and fully solved by E. Sawyer in 1982 in \cite{sawyer82b}. The general question is to find a necessary and sufficient condition for a  pair of unrelated weights $w$ and $\si$ for which  the following estimate holds
\begin{equation}\label{two-weight}
  \Norm{M (f\si)}{L^p(w)}
  \leq B\,\Norm{f}{L^p(\si)}.
\end{equation}
for a finite constant $B$. Then the main result of E.~Sawyer shows that this is the case if and only if there exists a finite $c$ such that  
$$
\int_{Q}
M(\si \chi_{_{Q}})(y)^{p}\, w(y)dy
\le c\, \si(Q)
$$
for all cubes  $Q$. Furthemore, it is shown in \cite{Moen:Two-weight} that if $B$ denotes the best constant then 
$$
B \approx  \sup_{Q} \left(  \frac{\int_{Q}M(\si \chi_{_{Q}})^{p}\, w \ud x}
{\si(Q)}\right)^{1/p}
$$

Since this condition is hard to verify in practice, the first author considered in \cite{perez95}
conditions closer in spirit to the classical two weight $A_p$ condition: 
\begin{equation*}
  A_p[w,\si]:=\sup_Q\Big(\avgint_Q w\Big)\Big(\avgint_Q \si\Big)^{p-1},
\end{equation*}
which reduces to $[w]_{A_p}$ if $\si=w^{-1/(p-1)}$. A consequence of the main result from \cite{perez95} establishes that if $\delta>0$ and the quantity
\begin{equation}\label{LlogLsuffCondition}
\sup_{Q} \, \Big( \avgint_Q w \Big)  \|\si\|^{p-1}_{ L(\log L)^{p-1+\delta},Q} 
\end{equation}
is finite then the two weight norm inequality \eqref{two-weight} holds. 
A recent result of the second author and M. Masty{\l}o \cite{MP} allows to go beyond condition \eqref{LlogLsuffCondition}
 and improve the main results from \cite{perez95}.

  In this paper we consider a different new quantity, namely
\begin{equation*}
  B_p[w,\si]:=\sup_Q\Big(\avgint_Q w\Big)\Big(\avgint_Q \si\Big)^{p}\exp\Big(\avgint_Q\log\si^{-1}\Big).
\end{equation*}
To understand this new quantity we observe that it is  simply the functional on $Q$ defining the $A_p[w,\si]$ condition multiplied by $\avgint_Q \si \exp\big(\avgint_Q\log\si^{-1}\big)\geq 1.$ Then it is immediate that
\begin{equation*}
  A_p[w,\si]\leq B_p[w,\si]\leq A_p[w,\si]A_\infty[\si],
\end{equation*}
the difference of the last two being that $A_p[w,\si]A_\infty[\si]$ involves two independent suprema, as opposed to just one in $B_p[u,v]$.

We will consider first the dyadic maximal operator $M_d$, for which we can prove a dimension-free bound. Let us also introduce the weighted dyadic  maximal function
\begin{equation*}
  M_{d,\si} f:=\sup_{Q\in\mathscr{D}}  
  \frac{\chi_Q}{ \sigma(Q)  }
\int_{Q}
|f(y)|\, \sigma(y)dy,
\end{equation*}
which controls $M_d(f\sigma)$ as follows:

\begin{theorem}\label{dyadic:sharpBuckley}
Let $p\in(1,\infty)$, then 
\begin{equation*}
\begin{split}
  \Norm{M_d (f\si)}{L^p(w)}
  &\leq 4e\cdot \big(B_p[w,\si]\big)^{1/p}\Norm{M_{d,\sigma}f}{L^p(\si)} \\
  &\leq 4e\cdot p'\cdot \big(B_p[w,\si]\big)^{1/p}\Norm{f}{L^p(\si)}
\end{split}
\end{equation*}
and also
\begin{equation*}
\begin{split}
  \Norm{M_d (f\si)}{L^p(w)}
  &\leq 4e\cdot \big([w]_{A_p}[\si]_{A_\infty}'\big)^{1/p}\Norm{M_{d,\sigma}f}{L^p(\si)} \\
  &\leq 4e\cdot p'\cdot \big([w]_{A_p}[\si]_{A_\infty}'\big)^{1/p}\Norm{f}{L^p(\si)}.
\end{split}
\end{equation*} 
\end{theorem}

The main estimate in both chains of inequalities is of course the first one, since the second is simply the universal estimate for the weighted dyadic maximal function on the weighted $L^p$ space with the same weight:
\begin{equation*}
  \|M_{d,\sigma}\|_{\mathscr{B}(L^p(\sigma))}\leq p'.
\end{equation*}
Obviously, in this dyadic version, it suffices to have the supremum in the weight constants over dyadic cubes only, and to only use the dyadic square function in the definition of $[\si]_{A_\infty}'$. And specializing to the case $\si=w^{-1/(p-1)}$, by the standard dual weight trick, we also get the bounds
\begin{equation*}
  \Norm{M_d f}{L^p(w)}
  \leq \begin{cases} 4e\cdot p'\cdot \big(B_p[w,w^{-1/(p-1)}]\big)^{1/p}\Norm{f}{L^p(w)},  & \\
                                   4e\cdot p'\cdot \big([w]_{A_p}[w^{-1/(p-1)}]_{A_\infty}'\big)^{1/p}\Norm{f}{L^p(w)}. & \end{cases}
\end{equation*}
Let us also recall how such dyadic bounds yield corresponding results for the Hardy--Littlewood maximal operator by a standard argument.

\begin{proof}[Proof of Theorem \ref{sharpBuckley}]
Consider the $2^d$ shifted dyadic systems
\begin{equation*}
  \mathscr{D}^{\alpha}:=\{2^{-k}\big([0,1)^d+m+(-1)^k\alpha\big):k\in\Z,m\in\Z^d\},\qquad\alpha\in\{0,\tfrac13\}^d.
\end{equation*}
One can check (perhaps best in dimension $n=1$ first) that any cube $Q$ is contained in a shifted dyadic cube $Q^{\alpha}\in\mathscr{D}^{\alpha}$ with $\ell(Q^{\alpha})\leq 6\ell(Q)$, for some $\alpha$. Hence
\begin{equation*}
  \avgint_Q\abs{f}\leq 6^d\avgint_{Q^\alpha}\abs{f}\leq 6^d M_d^{\alpha} f,
\end{equation*}
and therefore
\begin{equation*}
  Mf\leq 6^d\sum_{\alpha\in\{0,\frac13\}^d}M_d^{\alpha}f.
\end{equation*}
Thus the norm bound for $M_d$ may be multiplied by $12^d$ to give a bound for $M$.
\end{proof}

\begin{remark}
A recent result of the first author and A.~Kairema \cite{HytKai} allows to perform a similar trick with adjacent dyadic systems even in an abstract space of homogeneous type. Thus, Corollary~\ref{sharpBuckley} readily extends to this generality as well.
\end{remark}

\subsection{Proof of Theorem~\ref{dyadic:sharpBuckley}}

\

We start by observing that it suffices to have a uniform bound over all linearizations
\begin{equation*}
  \tilde{M}(f\si)=\sum_{Q\in\mathscr{D}}\chi_{E(Q)}\ave{f\si}_Q,
\end{equation*}
where the sets $E(Q)\subseteq Q$ are pairwise disjoint.  Here we using the following notation
$$
\ave{f}_Q=\avgint_Q f=  \avgint_Q f(x)\, \ud x
$$
and 
$$
\ave{f}_Q^{\si} =\frac{1}{\si(Q)}\int_Q f(x)\,\si(x) \ud x
$$
where as usual $\si(E)=\int_Q \si(x)\, \ud x$

By this disjointness,
\begin{equation*}
\begin{split}
  \Norm{\tilde{M}(f\si)}{L^p(w)}
  &=\Big(\sum_{Q\in\mathscr{D}} w(E(Q))\ave{f\si}_Q^p\Big)^{1/p} \\
  &=\Big(\sum_{Q\in\mathscr{D}} w(E(Q))\Big(\frac{\si(Q)}{\abs{Q}}\Big)^p(\ave{f}_Q^{\si})^p\Big)^{1/p}
\end{split}
\end{equation*}
Now recall:

\begin{theorem}[Dyadic Carleson embedding theorem]
Suppose that the nonnegative numbers $a_Q$ satisfy
\begin{equation*}
  \sum_{Q\subseteq R}a_Q
  \leq A\si(R)\qquad\forall R\in\mathscr{D}.
\end{equation*}
Then, for all $p\in[1,\infty)$ and $f\in L^p(\si)$,
\begin{equation*}
\begin{split}
  \Big(\sum_{Q\in\mathscr{D}}a_Q(\ave{f}_Q^{\si})^p\Big)^{1/p}
  &\leq A^{1/p}\Norm{M_{d,\sigma}f}{L^p(\si)} \\
  &\leq A^{1/p}\cdot p'\cdot\Norm{f}{L^p(\si)}\qquad\text{if}\quad p>1.
\end{split}
\end{equation*}
\end{theorem}

Since this is a slightly nonstandard formulation, although immediate by inspection of the usual argument, we provide a proof for completeness:

\begin{proof}
We view the sum $\sum_{Q}a_Q(\ave{f}_Q)^p$ as an integral on a measure space $(\mathscr{D},\mu)$ built over the set of dyadic cubes $\mathscr{D}$, assigning to each $Q\in\mathscr{D}$ the measure $a_Q$. Thus
\begin{equation*}
\begin{split}
  \sum_{Q\in\mathscr{D}}a_Q(\ave{f}_Q)^p
  &=\int_0^{\infty}p\lambda^{p-1}\mu(\{Q\in\mathscr{D}:\ave{f}_Q>\lambda\})\ud\lambda \\
  &=:\int_0^{\infty}p\lambda^{p-1}\mu(\mathscr{Q}_{\lambda})\ud\lambda
\end{split}
\end{equation*}
Let $\mathscr{Q}_{\lambda}^*$ be the set of maximal dyadic cubes $R$ with the property that $\ave{f}_R>\lambda$. The cubes $R\in\mathscr{Q}_{\lambda}^*$ are disjoint, and their union is equal to the set $\{M_{d,\si}f>\lambda\}$. Thus
\begin{equation*}
  \mu(\mathscr{Q}_{\lambda})
  =\sum_{Q\in\mathscr{Q}_{\lambda}}a_Q
  \leq\sum_{R\in\mathscr{Q}_{\lambda}^*}\sum_{Q\subseteq R}a_Q
  \leq\sum_{R\in\mathscr{Q}_{\lambda}^*}A\si(R)
  =A\si(M_{d,\si}f>\lambda),
\end{equation*}
and hence
\begin{equation*}
  \sum_{Q\in\mathscr{D}}a_Q(\ave{f}_Q)^p
  \leq A\int_0^{\infty}p\lambda^{p-1}\si(M_{d,\si}f>\lambda)\ud\lambda
  =A\|M_{d,\sigma}f\|_{L^p(\sigma)}^p.\qedhere
\end{equation*}
\end{proof}

If we apply the Carleson embedding with $a_Q=w(E(Q))\big(\si(Q)/\abs{Q}\big)^p$, we find that
\begin{equation}\label{eq:linM2weight}
  \Norm{\tilde{M}(f\si)}{L^p(w)}\leq A^{1/p}\Norm{M_{d,\sigma}f}{L^p(\si)}
\end{equation}
provided that
\begin{equation}\label{eq:testLinearM}
  \sum_{Q\subseteq R}w(E(Q))\Big(\frac{\si(Q)}{\abs{Q}}\Big)^p\leq A\,\si(R)\qquad\forall R\in\mathscr{D}.
\end{equation}
Note that on $E(Q)\subseteq Q\subseteq R$, we have $\si(Q)/\abs{Q}\leq M(\si \chi_R)$, and hence
\begin{equation*}
\begin{split}
    \sum_{Q\subseteq R}w(E(Q))\Big(\frac{\si(Q)}{\abs{Q}}\Big)^p
  &=\int\sum_{Q\subseteq R}\chi_{E(Q)}\Big(\frac{\si(Q)}{\abs{Q}}\Big)^p w \\
  &\leq\int\sum_{Q\subseteq R}\chi_{E(Q)} M(\chi_R\si)^p w
     \leq \int_R M(\chi_R\si)^p w.
\end{split}
\end{equation*}
So if $\Norm{\chi_R M(\chi_R \si)}{L^p(u)}\leq A^{1/p} \,\si(R)^{1/p}$, then \eqref{eq:testLinearM} holds, hence by Carleson's embedding also \eqref{eq:linM2weight}, and therefore the original two-weight inequality
\begin{equation*}
\Norm{M(f\si)}{L^p(u)}\leq A^{1/p}\,\Norm{M_{d,\si}f}{L^p(\si)}.
\end{equation*}
Hence, we are reduced to proving that
\begin{equation}\label{eq:M2prove}
\Norm{\chi_R M(\chi_R \si)}{L^p(u)}^p \leq A\,\si(R),\qquad A=(4e)^{1/p}\cdot B_p[w,\si]. 
\end{equation}
(In fact, the argument up to this point was essentially reproving Sawyer's two-weight characterization for the maximal function, paying attention to the constants.)

To prove~\eqref{eq:M2prove}, we exploit another linearization of $M$ involving the \emph{principal cubes}, as in the proof of the $A_2$ theorem: Let $\mathscr{S}_0:=\{R\}$ and recursively
\begin{equation*}
  \mathscr{S}_k:=\bigcup_{S\in\mathscr{S}_{k-1}}\{Q\subset S:\ave{\si}_Q>2\ave{\si}_S,Q\text{ is a maximal such cube}\},
\end{equation*}
and then $\mathscr{S}:=\bigcup_{k=0}^{\infty}\mathscr{S}_k$. The pairwise disjoint subsets $E(S)\subseteq S$, defined in \eqref{eq:ESdef}, satisfy $\abs{E(S)}\geq\tfrac12\abs{S}$ by \eqref{eq:ESest}, and they partition $R$.

 If $x\in E(S)$ and $Q\owns x$, then $\ave{\si}_Q\leq 2\ave{\si}_S$, and hence $\chi_R M(\chi_R\si)\leq 2\ave{\si}_S$ on $\chi_{E(S)}$. So altogether
\begin{equation}\label{eq:BuckleyKeyStep}
\begin{split}
  \Norm{\chi_R M(\chi_R\si)}{L^p(w)}^p
  &\leq 2^p\BNorm{\sum_{S\in\mathscr{S}}\chi_{E(S)}\ave{\si}_S}{L^p(w)}^p \\
  &=2^p \sum_{S\in\mathscr{S}} w(E(S))\Big(\frac{\si(S)}{\abs{S}}\Big)^p \\
  &\leq 2^p\sum_{S\in\mathscr{S}} \frac{w(S)}{\abs{S}}
      \Big(\frac{\si(S)}{\abs{S}}\Big)^p\abs{S} \\
  &\leq 2^{p+1}\sum_{S\in\mathscr{S}} B_p[w,\si]\exp\Big(\avgint_S\log\si\Big)\abs{E(S)} \\ 
  &\leq 2^{p+1}B_p[w,\si]\int_R \sum_{S\in\mathscr{S}} \exp\Big(\avgint_S\log\si\Big)\chi_{E(S)} \\
  &\leq 2^{p+1}B_p[w,\si]\int_R \sup_{Q\in\mathscr{D}} \chi_Q\exp\Big(\avgint_Q\log\si \chi_R\Big)\\
  &= 2^{p+1}B_p[w,\si]\int_R M_0(\chi_R\si),
\end{split}
\end{equation}
where $M_0$ is the (dyadic) logarithmic maximal function introduced in Lemma~\ref{lem:logMaxFn}. By this lemma, we then have
\begin{equation*}
  \Norm{\chi_R M(\chi_R\si)}{L^p(u)}^p
  \leq 4^p B_p[w,\si]\cdot e\cdot\si(R),
\end{equation*}
which proves \eqref{eq:M2prove}, and hence Theorem~\ref{dyadic:sharpBuckley}, upon taking the $p$th root.

\

In order to prove the second version of Theorem~\ref{dyadic:sharpBuckley}, we only need to make a slight modification in the estimate~\eqref{eq:BuckleyKeyStep}. We then compute:
\begin{equation*}
\begin{split}
  \Norm{\chi_R M(\chi_R\si)}{L^p(w)}^p
  &\leq 2^p\sum_{S\in\mathscr{S}} \frac{w(S)}{\abs{S}}
      \Big(\frac{\si(S)}{\abs{S}}\Big)^p\abs{S} \\
  &\leq 2^{p+1}\sum_{S\in\mathscr{S}} [w]_{A_p}\frac{\si(S)}{\abs{S}}\abs{E(S)} \\ 
  &\leq 2^{p+1}[w]_{A_p}\sum_{S\in\mathscr{S}}\int_{E(S)} M(\si\chi_Q) \\
  &= 2^{p+1}[w]_{A_p}\int_Q M(\si\chi_Q) \\
  &= 2^{p+1}[w]_{A_p}[\si]_{A_\infty}'  \si(Q),
\end{split}
\end{equation*}
by a direct application of the definition of $[\si]_{A_\infty}'$ in the last step, and this completes the alternative argument.

\subsection{Another proof of Theorem~\ref{dyadic:sharpBuckley}}\label{sec:2ndProofBuckley}

\

We finish this section by providing yet another proof variant for Theorem~\ref{dyadic:sharpBuckley}. This proof is more elementary, since it does not need the reduction to the testing condition~\eqref{eq:M2prove}, and it uses the more standard Calder\'on--Zygmund-type stopping cubes, instead of the principal cubes. Its disadvantage is the fact the we cannot recover the dimension-independence by this argument. On the other hand, the proof may be extended to maximal functions defined in term of a general basis (see \cite{GCRdF} Section IV.4).

\begin{proof}[A simpler proof of Theorem~\ref{dyadic:sharpBuckley} with a dimension-dependent bound]
Fix $a> 2^d$. For each integer $k$ let
$$\Omega_{k} = \{x\in
{\mathbb R}^d : M_d( f\,\si )(x)>a^k \}.
$$
By standard arguments we consider the 
Calder\'{o}n--Zygmund decomposition  and  there is a family of maximal non-overlapping
dyadic cubes $\{Q_{k,j}\}$ for which $\Omega_{k} = \bigcup_{j}Q_{k,j}$
and
\begin{equation}
a^{k} <
\frac{1}{|Q_{k,j}|}\int_{Q_{k,j}}|f(y)|\,\si(y)dy
 \le 2^da^{k}.
\label{ulti}
\end{equation}
Now, 
\begin{equation*}
\begin{split}
  \int_{{\mathbb R}^d} &M_d( f\si)^{p}\,w\ud x
   = \sum_{k} \int_{ \Omega_{k}\setminus\Omega_{k+1}}M_d( f\si\,)^{p}\,w\ud x \\
  &\le a^{p} \sum_{k} a^{kp}w(\Omega_{k})=a^p\sum_{k,j}a^{kp} w( Q_{k,j})\\
  &\le a^p\sum_{k,j} \left(\frac{1}{|Q_{k,j}|}\int_{Q_{k,j}}|f(y)|\,\si(y)\ud y\right)^{p}w(Q_{k,j})\\
  &=a^p\sum_{k,j} \big(\ave{|f|}_{Q_{k,j}}^{\si}\big)^p\left(\frac{\si(Q_{k,j})}{  |Q_{k,j}| } \right)^{p}w(Q_{k,j})\\
  &\leq a^p B_p[w,\si]\, \sum_{k,j} \big(\ave{|f|}_{Q_{k,j}}^{\si}\big)^p\, |Q_{k,j}| \exp\Big(\avgint_{Q_{k,j}} \log \sigma(t)\ud t\Big)\\
  &=a^p B_p[w,\si]\, \sum_{Q\in\mathscr{D}}\big(\ave{|f|}_{Q}^{\si}\big)^p a_Q,
\end{split}
\end{equation*}
where
\begin{equation*}
  a_Q=\begin{cases} |Q| \exp\big(\avgint_{Q} \log \sigma\big) & \text{if $Q=Q_{k,j}$ for some $(k,j)$}, \\ 0 & \text{else}. \end{cases}
\end{equation*}
By the dyadic Carleson embedding theorem, we can hence conclude that
\begin{equation*}
   \int_{{\mathbb R}^d} M^d( f\si)^{p}\,w\ud x\leq a^pB_p[w,\sigma]A\int_{\mathbb{R}^d}(M_{d,\sigma}f)^p\sigma\ud x,
\end{equation*}
provided that we check the condition
\begin{equation}\label{eq:Carleson2check}
  \sum_{Q\subseteq R}a_Q=\sum_{k,j:Q_{k,j}\subseteq R}|Q_{k,j}| \exp\Big(\avgint_{Q_{k,j}} \log \sigma \Big)\leq A|R|.
\end{equation}

To estimate the left side of \eqref{eq:Carleson2check}, we do first the following: For each $(k,j)$ we set $E_{k,j} = Q_{k,j}\setminus \Omega_{k+1}$. 
Observe that the sets of the family $E_{k,j}$ are pairwise disjoint. We claim 
that 
\begin{equation}\label{claim}
|Q_{k,j}|< \frac{a}{a-2^d}\,|E_{k,j}|
\end{equation}
for each $k,j$. Indeed, by (\ref{ulti}) and  H\"older's inequality,
\begin{eqnarray*}
|Q_{k,j}\cap \Omega_{k+1}|&=&\sum_{Q_{k+1,l}\subset
Q_{k,j}}|Q_{k+1,l}|\\
&<& \frac{1}{a^{k+1 } }\sum_{Q_{k+1,l}\subset
Q_{k,j}}  \int_{Q_{k+1,l}}|f|\,\si \\
&\le& \frac{1}{a^{k+1}}
 \int_{Q_{ k,j}}|f|\,\si \le
 \frac{2^d}{a}|Q_{k,j}|,
\end{eqnarray*}
which proves \eqref{claim}. With  $\be=\frac{a}{a-2^d}$ we can estimate the left side of \eqref{eq:Carleson2check} as follows: 

\begin{equation*}
\begin{split}
  \sum_{Q\subseteq R}a_Q
  &\leq \beta \sum_{  (k,j):Q_{k,j}\subseteq R} |E_{k,j}|  \exp\Big(\avgint_{Q_{k,j}}  \log \sigma(t)dt\Big) \\
  &\leq \beta \sum_{  (k,j):Q_{k,j}\subseteq R}\int_{E_{k,j}} M_0(\sigma 1_R)(x)\ud x \\
  &\leq \beta \int_{R} M_0(\sigma 1_R)(x)\ud x \leq\beta\, e\,\sigma(R),
\end{split}
\end{equation*}
where we used the definition and the $L^1$ boundedness of the logarithmic dyadic maximal function. This proves \eqref{eq:Carleson2check} with $A=\beta\,e$, concluding the proof.
\end{proof}

\section{Proof of the extrapolation theorems} \label{extrapolation}

We will prove in this section the Upper and Lower Extrapolation Theorems  \ref{thm:lowerExtrap} and \ref{thm:upperExtrap}. Recall that the initial hypothesis is given by the expression: 
\begin{equation*}
  \Norm{Tf}{L^r(w)}
  \leq\varphi\big([w]_{A_r},[w]_{A_{\infty}},[w^{-1/(r-1)}]_{A_{\infty}}^{(r-1)}\big)\Norm{f}{L^r(w)}
\end{equation*}
for some $r\in (1,\infty)$.

\begin{proof}[Proof of Theorem~\ref{thm:lowerExtrap}]
Our argument is modeled after a simplified proof of the Dragi\v{c}evi\'c--Grafakos--Pereyra--Petermichl \cite{DGPP} result due to Duoandikoetxea  \cite{Duo-JFA} (see also \cite{CMP-book}).

Fix some $p\in(1,r)$, $w\in A_p$, $f\in L^p(w)$ and $g:=\abs{f}/\Norm{f}{L^p(w)}$. Let
\begin{equation*}
   Rg:=\sum_{k=0}^{\infty}\frac{2^{-k}M^k g}{\Norm{M}{\mathscr{B}(L^p(w))}^k}
\end{equation*}
so that
\begin{equation*}
  \abs{g}\leq Rg,\qquad\Norm{Rg}{L^p(w)}\leq 2\Norm{g}{L^p(w)}=2,\qquad
  [Rg]_{A_1}\leq 2\Norm{M}{L^p(w)}.
\end{equation*}
Then by H\"older's inequality
\begin{equation*}
\begin{split}
  \Norm{Tf}{L^p(w)}
  &=\Big(\int\abs{Tf}^p (Rg)^{-(r-p)p/r}(Rg)^{(r-p)p/r}w\Big)^{1/p} \\
  &\leq\Big(\int\abs{Tf}^r (Rg)^{-(r-p)}w\Big)^{1/r}\Big(\int (Rg)^{p}w\Big)^{1/p-1/r} \\
  &\leq\Norm{Tf}{L^r(W)}(2^p)^{1/p-1/r} \leq 2\Norm{Tf}{L^r(W)}, \qquad W:=(Rg)^{-(r-p)}w.
\end{split}
\end{equation*}
By assumption, we have
\begin{equation*}
  \Norm{Tf}{L^r(W)}
  \leq\varphi\big(\big([W]_{A_r},[W]_{A_{\infty}},[W^{-1/(r-1)}]_{A_{\infty}}^{(r-1)}\big)\Norm{f}{L^r(W)}
\end{equation*}
where
\begin{equation*}
\begin{split}
  \Norm{f}{L^r(W)}
  &=\Big(\int\abs{f}^r (Rf)^{-(r-p)}w\Big)^{1/r}\Norm{f}{L^p(w)}^{(r-p)/r} \\
  &\leq \Big(\int\abs{f}^r \abs{f}^{-(r-p)}w\Big)^{1/r}\Norm{f}{L^p(w)}^{(r-p)/r}
  =\Norm{f}{L^p(w)},
\end{split}
\end{equation*}
so it remains to estimate the weight constants
\begin{equation*}
  [W]_{A_r},\qquad [W]_{A_{\infty}},\qquad [W^{-1/(r-1)}]_{A_{\infty}}.
\end{equation*}

Using $\sup_Q(Rg)^{-1}\leq[Rg]_{A_1}\ave{Rg}_Q^{-1}$ or H\"older's or Jensen's inequality where appropriate, we compute
\begin{equation*}
\begin{split}
  \ave{W}_Q &=\ave{(Rg)^{-(r-p)}w}_Q \\
    &\leq[Rg]_{A_1}^{r-p}\ave{Rg}_Q^{-(r-p)}\ave{w}_Q,\\
  \ave{W^{-1/(r-1)}}_Q^{r-1}
    &=\ave{(Rg)^{(r-p)/(r-1)}w^{-1/(r-1)}}_Q^{r-1} \\
    &\leq\ave{Rg}_Q^{r-p}\ave{w^{-1/(p-1)}}_Q^{p-1} ,\\
  \exp\ave{-\log W}_Q &=  \big(\exp\ave{\log(Rg)}_Q\big)^{r-p}\exp\ave{-\log w}_Q \\
      &\leq\ave{Rg}_Q^{r-p}\exp\ave{-\log w}_Q,
\end{split}
\end{equation*}
and
\begin{equation*}
\begin{split}
  \big(\exp &\ave{-\log W^{-1/(r-1)}}_Q\big)^{r-1} \\
    &=\big(\exp\ave{\log(Rg)^{-1}}_Q\big)^{r-p}\big(\exp\ave{-\log w^{-1/(r-1)}}_Q\big)^{r-1} \\
    &\leq[Rg]_{A_1}^{r-p}\ave{Rg}^{-(r-p)}\big(\exp\ave{-\log w^{-1/(p-1)}}_Q\big)^{p-1}.
\end{split}
\end{equation*}
Multiplying the appropriate estimates and using the definition, we then have
\begin{equation*}
\begin{split}
  &[W]_{A_r}\leq[Rg]_{A_1}^{r-p}[w]_{A_p},\qquad
  [W]_{A_\infty}\leq[Rg]_{A_1}^{r-p}[w]_{A_\infty},\\
  &[W^{-1/(r-1)}]_{A_{\infty}}^{r-1}\leq[Rg]_{A_1}^{r-p}[w^{-1/(p-1)}]_{A_\infty}^{p-1}.
\end{split}
\end{equation*}
(We do not know whether it is possible to make similar estimates for $[W]_{A_\infty}'$ in terms of $[w]_{A_\infty}'$; this is the reason why we need to use the $[\ ]_{A_\infty}$ constants in this proof.)

Next, recall that
\begin{equation*}
  [Rg]_{A_1}\leq 2\,  \Norm{M}{\mathscr{B}(L^p(w))}  
  \leq c_d\cdot p'\cdot [w]_{A_p}^{1/p} ([w^{-1/(p-1)}]_{A_\infty}')^{1/p}.
\end{equation*}
Thus we conclude the proof with
\begin{equation*}
\begin{split}
  \Norm{Tf}{L^p(w)}
  &\leq 2\Norm{Tf}{L^r(W)} 
    \leq 2\varphi\big([W]_{A_r},[W]_{A_{\infty}},[W^{-1/(r-1)}]_{A_{\infty}}^{(r-1)}\big)\Norm{f}{L^r(W)}\\
  &\leq 2\varphi\Big([Rg]_{A_1}^{r-p}\big([w]_{A_p},[w]_{A_{\infty}},[w^{-1/(p-1)}]_{A_{\infty}}^{(p-1)}\big)\Big)\Norm{f}{L^p(w)}\\
  &\leq 2\varphi\Big(2^{r-p}\Norm{M}{\mathscr{B}(L^p(w))}^{r-p}\big([w]_{A_p},[w]_{A_{\infty}},[w^{-1/(p-1)}]_{A_{\infty}}^{(p-1)}\big)\Big)\Norm{f}{L^p(w)}.\qedhere
\end{split}
\end{equation*}
\end{proof}

\begin{proof}[Proof of Theorem~\ref{thm:upperExtrap}]
Again, our argument is inspired by a simplified proof of the Dragi\v{c}evi\'c--Grafakos--Pereyra--Petermichl \cite{DGPP} result due to Duoandikoetxea  \cite{Duo-JFA} (see also \cite{CMP-book}).

Fix some $p\in(r,\infty)$, $w\in A_p$, $f\in L^p(w)$. By duality, we need have
\begin{equation*}
  \Norm{Tf}{L^p(w)}=\sup_{\substack{h\geq 0\\ \Norm{h}{L^{p'}(w)=1}}}\int\abs{Tf}hw.
\end{equation*}
We fix one such $h$, and try to bound the expression on the right.

Observe that the pointwise multiplication operators
\begin{equation*}
  h\mapsto wh: L^{p'}(w)\to L^{p'}(w^{1-p'}),\qquad
  g\mapsto\frac{1}{w}g: L^{p'}(w^{1-p'})\to L^{p'}(w)
\end{equation*}
are isometric. Let $R$ be as in the previous proof, except with $p'$ and $\si=w^{1-p'}$ in place of $p$ and $w$:
\begin{equation*}
  Rg:=\sum_{k=0}^{\infty}\frac{2^{-k}M^k g}{\Norm{M}{\mathscr{B}(L^{p'}(\si))}^k},
\end{equation*}
and $R'h:=w^{-1}R(wh)$. Then
\begin{equation*}
  h\leq R'h,\qquad\Norm{R'h}{L^{p'}(w)}\leq 2\Norm{h}{L^{p'}(w)}=2,\qquad
  [w R'h]_{A_1}\leq 2\Norm{M}{\mathscr{B}(L^{p'}(\si))}.
\end{equation*}
Then by H\"older's inequality
\begin{equation*}
\begin{split}
  \int\abs{Tf}hw &\leq\int\abs{Tf}(R'h)w
  =\int\abs{Tf}(R'h)^{(p-r)/[r(p-1)]}(R'h)^{(r-1)p/[r(p-1)]}w \\
  &\leq\Big(\int\abs{Tf}^r (R'h)^{(p-r)/(p-1)}w\Big)^{1/r}\Big(\int (R'h)^{p/(p-1)}w\Big)^{1/r'} \\
  &\leq\Norm{Tf}{L^r(W)}2^{p'/r'},\qquad W:=(R'h)^{(p-r)/(p-1)}w.
\end{split}
\end{equation*}

By assumption,
\begin{equation}\label{eq:upperExtAss}
  \Norm{Tf}{L^r(W)}
  \leq\varphi\left([W]_{A_r},[W]_{A_\infty},[W^{-1/(r-1)}]_{A_\infty}^{(r-1)}\right)\Norm{f}{L^r(W)} 
\end{equation}
where, by H\"older's inequality with exponents $p/r$ and its $p/(p-r)$,
\begin{equation*}
\begin{split}
  \Norm{f}{L^r(W)}
  &=\Big(\int\abs{f}^r w^{r/p}\cdot (R'h)^{(p-r)/(p-1)}w^{(p-r)/p}\Big)^{1/r} \\
  &\leq\Big(\int\abs{f}^p w\Big)^{1/p}\Big(\int (R'h)^{p/(p-1)}w\Big)^{1/r-1/p}
  \leq\Norm{f}{L^p(w)}(2^{p'})^{1/r-1/p},
\end{split}
\end{equation*}
so altogether, supressing the arguments of $\varphi$ from \eqref{eq:upperExtAss},
\begin{equation*}
\begin{split}
  \int\abs{Tf}hw
  &\leq\Norm{Tf}{L^r(W)}2^{p'/r'}
   \leq\varphi\big(\ldots\big)\Norm{f}{L^r(W)}2^{p'/r'} \\
  &\leq\varphi\big(\ldots\big)(2^{p'})^{1/r-1/p}\Norm{f}{L^p(w)}2^{p'/r'}
    =2\varphi(\ldots)\Norm{f}{L^p(w)}.
\end{split}
\end{equation*}

It remains to estimate
\begin{equation*}
   [W]_{A_r},\qquad [W]_{A_\infty},\qquad [W^{-1/(r-1)}]_{A_\infty}^{(r-1)}
\end{equation*}
for
\begin{equation*}
  W=(R'h)^{(p-r)/(p-1)}w=[(R'h)w]^{(p-r)/(p-1)}w^{(r-1)/(p-1)}.
\end{equation*}

We thus compute
\begin{equation*}
\begin{split}
  \ave{W}_Q &=\ave{(R'h)^{(p-r)/(p-1)}w}_Q \\
    &\leq\ave{(R'h)w}_Q^{(p-r)/(p-1)}\ave{w}_Q^{(r-1)/(p-1)}, \\
   \ave{W^{-1/(r-1)}}_Q^{r-1} 
   &=  \ave{(wR'h)^{-(p-r)/[(p-1)(r-1)]}w^{-1/(p-1)}}_Q^{r-1} \\
   &\leq [wR'h]_{A_1}^{(p-r)/(p-1)}\ave{(R'h)w}_Q^{-(p-r)/(p-1)}\ave{w^{-1/(p-1)}}_Q^{r-1}, \\
   \exp(-\ave{\log W}_Q)
   &=\big(\exp\ave{\log(wR'h)^{-1}}_Q\big)^{(p-r)/(r-1)}\big(\exp\ave{-\log w}_Q\big)^{(r-1)/(p-1)} \\
   &\leq[(R'h)w]_{A_1}^{(p-r)/(r-1)}\ave{(R'h)w}_Q^{-(p-r)/(r-1)} \\
   &\qquad\times\big(\exp\ave{-\log w}_Q\big)^{(r-1)/(p-1)},
\end{split}
\end{equation*}
and
\begin{equation*}
\begin{split}
    \big( \exp(- &\ave{\log W^{-1/(r-1)}}_Q)\big)^{r-1} \\
    &=\big(\exp(\ave{\log (wR'h)}_Q)\big)^{(p-r)/(p-1)}\big(\exp\ave{-\log w^{-1/(p-1)}}_Q\big)^{r-1} \\
    &\leq\ave{(R'h)w}_Q^{(p-r)/(r-1)}\big(\exp\ave{-\log w^{-1/(p-1)}}_Q\big)^{r-1}.
\end{split}
\end{equation*}
Multiplying the relevant quantities, it follows that
\begin{equation*}
\begin{split}
  [W]_{A_r} &\leq[(R'h)w]_{A_1}^{(p-r)/(p-1)}[w]_{A_p}^{(r-1)/(p-1)}, \\
  [W]_{A_\infty} &\leq[(R'h)w]_{A_1}^{(p-r)/(p-1)}[w]_{A_\infty}^{(r-1)/(p-1)}, \\
  [W^{-1/(r-1)}]_{A_\infty}^{r-1} &\leq[(R'h)w]_{A_1}^{(p-r)/(p-1)}[w^{-1/(p-1)}]_{A_\infty}^{(r-1)}, \\
\end{split}
\end{equation*}
Also recall that 
\begin{equation*}
  [(R'h)w]_{A_1}\leq 2\Norm{M}{\mathscr{B}(L^{p'}(w^{1-p'}))}
  \leq c_d[w^{1-p'}]_{A_{p'}}^{1/p'}[w]_{A_\infty}^{1/p'}
  =c_d[w]_{A_{p}}^{1/p}[w]_{A_\infty}^{1/p'},
\end{equation*}
and thus we conclude with
\begin{equation*}
\begin{split}
  &\Norm{Tf}{L^p(w)}
    \leq\int\abs{Tf}hw 
    \leq 2\varphi\left([W]_{A_r},[W]_{A_\infty},[W^{-1/(r-1)}]_{A_\infty}^{(r-1)} \right)\Norm{f}{L^p(w)} \\
  &\leq 2\varphi\left([(R'h)w]_{A_1}^{(p-r)/(p-1)}\big([W]_{A_r},[W]_{A_\infty},[W^{-1/(r-1)}]_{A_\infty}^{(r-1)}\big)\right)\Norm{f}{L^p(w)} \\
  &\leq 2\varphi\left((2\Norm{M}{\mathscr{B}(L^{p'}(w^{1-p'}))})^{(p-r)/(p-1)}\right.\\
    &\qquad\qquad\times\left.\big([w]_{A_r}^{(r-1)/(p-1)},[w]_{A_\infty}^{(r-1)/(p-1)},
         [w^{-1/(p-1)}]_{A_\infty}^{(r-1)}\big)\right)\Norm{f}{L^p(w)}.\qedhere
\end{split}
\end{equation*} 
\end{proof}

\section{The $A_1$ theory, proof of Theorem \ref{thm:A1strong(p,p)} and its consequences}  \label{A1theory}

\subsection{ The main lemma.}

\

The proofs of the theorems will be based on the following lemma.

\begin{lemma} \label{keylemma}
Let $T$ be any Calder\'{o}n-Zygmund singular integral operator and
let $w$ be any weight. Also  let \, $p,r \in (1,\infty)$.\,Then, there is a constant $c=c_{d,T}$ such that:
$$
\|Tf\|_{\strt{1.7ex}L^p(w)}\le cpp'\,(r')
^{1/p'}\|f\|_{\strt{1.7ex}L^p(M_rw)}
$$
where as usual we denote  $M_rw=M(w^r)^{1/r}$.
\end{lemma}

This is a consequence of the following estimate that can be found in \cite{LOP3} when $r \in (1,2]$:
$$
\|Tf\|_{\strt{1.7ex}L^p(w)}\le cpp'\,\Big(\frac{1}{r-1}\Big)^{1-1/pr}\,\|f\|_{\strt{1.7ex}L^p(M_rw)}
$$
since
$$
\Big(\frac{1}{r-1}\Big)^{1-1/pr}\leq (r')^{1-1/p+1/pr'} \leq 2
(r')^{1/p'}
$$
and $t^{1/t}\le 2$, $t\ge 1$.

\subsection{ Proof of the sharp reverse H\"older's inequality.}\label{sec:RHI}

\

We need the following lemma:

\begin{lemma} \label{llogllema}
For any cube $Q$ and any measurable function $w$,
\begin{equation}\label{logest}
\int_Q w\, \log(e +
\frac{w}{\ave{w}_Q})\ud x  \le 2^{d+1}
\int_{Q} M(w\chi_Q)\ud x,
\end{equation}
Hence,  if  $w\in A_{\infty}$
\begin{equation} \label{lLlogL-infty}
\sup_Q\frac{1}{w(Q)}\int_Q w(y) \log(e+\frac{w(y)}{\ave{w}_Q} )\, dy \leq 2^{d+1}\, [w]_{A_\infty}'
\end{equation}

\end{lemma}

The essential idea of the proof can be traced back to the well  known  $L \log L$ estimate for $M$  in \cite{Stein-LlogL}. However these estimates are not homogeneus. A proof of this lemma within the context of spaces of homogeneous type can essentially be found in \cite[Lemma 8.5]{PW-JFA}    (see also  \cite[p.~17, inequality (2.15)]{Wilson-LNM} for a different proof).

\begin{proof}[Proof of Lemma \ref{llogllema}]
Fix a cube $Q$.  By homogeneity we assume that $\ave{w}_Q=1$. The key estimate follows from the ``reverse weak
type $(1,1)$ estimate'': if $w$ is nonnegative and $t>\ave{w}_Q$,
\begin{equation}
\frac{1}{t}\int_{\{x\in Q : w(x)> t\}} w\ud x \le 2^d\, |\{x\in Q:
M(w\chi_Q)(x)> t \}|. \label{localweak11backwards}
\end{equation}
Now,
\begin{equation*}
\begin{split}
  \frac{1}{|Q|} \int_Q w\log(e + w)\ud x
  &= \frac{1}{|Q|}\int_{0}^{\infty}\frac{1}{e+t}w(\{x\in Q : w(x) > t\}) \ud t \\
  &= \frac{1}{|Q|}\int_{0}^{1} + \, \frac{1}{|Q|}\int_{ 1}^{\infty} \cdots = I+II,
\end{split}
\end{equation*}
and
$$
I\le 1\leq  \frac{1}{|Q|}\,\int_{Q} M(w\chi_Q)\ud x.
$$
For $II$ we use estimate (\ref{localweak11backwards}):
\begin{equation*}
\begin{split}
  II &= \frac{1}{|Q|} \int_{1}^{\infty} \frac{1}{e+t} w(\{x \in Q: w(x) > t \})\ud t \\
  &\le \frac{2^d}{|Q|} \int_{1}^{\infty} \frac{t}{e+t} |\{x\in Q :M(w\chi_Q)(x) > t\}|\ud t \\
  & \le \frac{2^d}{|Q|}\int_{0}^{\infty} |\{x \in {Q} : M(w\chi_Q)(x) >t\}|\ud t \\
  &= \frac{2^d}{|Q|}\int_{Q} M(w\chi_Q)(x)\,\ud x.
\end{split}
\end{equation*}
This gives \eqref{logest} and  \eqref{lLlogL-infty} follows from the definiton of $[w]_{A_\infty}'$.
\end{proof}

The main use of the Lemma is the following key observation that we borrow from \cite{Wilson-LNM}, p.~45:

\begin{lemma}
Let $S \subset Q$ and let $\la>0$, then 
\begin{equation} \label{A-inftyproperty}
\frac{|S|}{|Q|}<e^{-\la}  \quad\mbox{implies}\quad  \frac{w(S)}{w(Q)}<  \frac{2^{d+2} [w]_{A_\infty}'}{\la}   + e^{-\la/2}
\end{equation}
\end{lemma}  

\begin{proof}
Indeed, if  \,$E_{\la}=\{x\in Q: w(x)>e^{\la}\ave{w}_Q \}$\, then \,$w(E_{\la}) \leq \frac{2^{d+1}}{\la}w(Q)$\, by  \eqref{lLlogL-infty}. Therefore: 
\begin{equation*}
\begin{split}
  w(S) &\leq  w(S\cap {E_{\la/2}}) +  w(S\setminus {E_{\la/2}} ) \leq \frac{2^{d+2}\,[w]_{A_\infty}'}{\la}\,w(Q)+ e^{\la/2}\ave{w}_Q\,|S| \\
  &\leq \frac{2^{d+2}\,[w]_{A_\infty}'}{\la}\,w(Q)+ e^{\la/2} e^{-\la} \, w(Q)\,  \qquad \mbox{by the hypothesis in \eqref{A-inftyproperty}} \\
  &=\frac{2^{d+2}\,[w]_{A_\infty}'}{\la}\,w(Q)+ e^{-\la/2} \, w(Q)\,
\end{split}
\end{equation*}
and this proves the claim \eqref{A-inftyproperty}.
\end{proof}

\begin{proof}  [Proof of  Theorem~\ref{thm:SharpRHI}]
Recall that we have to prove that 
\begin{equation*}
  \Big(\avgint_Q w^{r(w)}\Big)^{1/r(w)}\leq 2\avgint_Q w.
\end{equation*}
where 
$$r(w):=1+\frac{1}{ \tau_d\,[w]_{A_{\infty}}'},$$ 
and where $\tau_d$ is a large dimensional constant. 

Observe that by homogeneity we can assume that $\avgint_{Q} w=1$. We use the dyadic maximal function on the dyadic subcubes of a given $Q$:
\begin{equation*}
\begin{split}
  \int_{Q}w^{1+\eps}
  &\leq\int_{Q}M_d(w \chi_{Q})^{\eps}w 
    =\int_0^{\infty}\eps t^{\eps-1}w(\{x \in Q:M_d(w\chi_{Q})>t\})\ud t. \\
  &\leq\int_0^{1}\eps t^{\eps-1}w(Q)\ud t
    +\eps \int_1^{\infty}\eps t^{\eps}w(\{x \in Q:M_d(w\chi_{Q})>t\})\frac{dt}{t} \\
  &\leq |Q|+ \eps \sum_{k \geq 0}  \int_{a^k}^{a^{k+1}}t^{\eps}w(\{x \in Q:M_d(w\chi_{Q})>t\})\frac{dt}{t} \\
  &\leq |Q|  + \eps a^{\eps}\sum_{k\geq 0}a^{k\eps}\,  \int_{a^k}^{a^{k+1}}w(\{x \in Q:M_d(w\chi_{Q})>a^k\})\frac{dt}{t}, \quad \mbox{for} \,a\gg 1,  \\
  &=|Q|  + \eps a^{\eps}\,\log a \sum_{k\geq 0}a^{k\eps}\,  w(\Omega_k)
\end{split}
\end{equation*}
where 
$$\Omega_{k} = \{x\in Q: M_d(w\chi_{Q}(x)>a^k \}.
$$

Since $a^k\geq 1=\avgint_{Q} w$ we can consider the 
Calder\'{o}n--Zygmund decomposition  $w$ adapted to $Q$.  There is a family of maximal non-overlapping
dyadic cubes $\{Q_{k,j}\}$ strictly contained in $Q$ for which $\Omega_{k} = \bigcup_{j}Q_{k,j}$
and
\begin{equation}
a^{k} <
\avgint_{Q_{k,j}} w
 \le 2^da^{k}.
\label{LocalCZi}
\end{equation}
Now, 
$$
\sum_{k\geq 0}a^{k\eps}\,  w(\Omega_k)
=
\sum_{k,j}a^{k\eps} w( Q_{k,j}) \leq \sum_{k,j} \left(\frac{1}{|Q_{k,j}|}\int_{Q_{k,j}}w(y)\ud y\right)^{\eps}w(Q_{k,j})\\
$$

To now need to estimate $w(Q_{k,j})$  and we pursue similarly to Section~\ref{sec:2ndProofBuckley}, see in particular \eqref{claim}: For each $(k,j)$ we set $E_{k,j} = Q_{k,j}\setminus \Omega_{k+1}$. 
Observe that the sets of the family $E_{k,j}$ are pairwise disjoint. But exactly as in \eqref{claim} we have that for $a>2^d$ and for each $k,j$:
\begin{equation}\label{claim2}
|Q_{k,j}|< \frac{a}{a-2^d}\,|E_{k,j}|.
\end{equation}

I removed the repetition from Section~\ref{sec:2ndProofBuckley}.

We now apply  \eqref{A-inftyproperty} with $Q=Q_{k,j}$ and $S=Q_{k,j}\cap \Omega_{k+1}$. Choose 
$\la$ such that $e^{-\la}= \frac{2^d}{a}$, namely $\la=\log \frac{a}{2^d}$. Then applying \eqref{A-inftyproperty}  we have that 
$$
\frac{w(Q_{k,j}\cap \Omega_{k+1})}{w(Q_{k,j})}<  \frac{2^{d+2}\, [w]_{A_\infty}'}{\log \frac{a}{2^d}}   + (\frac{2^d}{a})^{1/2}.
$$
Since $a>2^d$ is available we choose $a=2^de^{ L\,[w]_{A_\infty}'}$, with $L$ a large dimensional constant to be chosen.  If in particular   $L\geq 2^{d+4}$ we have 
$$
\frac{w(Q_{k,j}\cap \Omega_{k+1})}{w(Q_{k,j})}<  \frac{2^{d+2}}{L}  + e^{-\,[w]_{A_\infty}'L/2} <\frac14+\frac14 = \frac12
$$
This yields that \,$w(Q_{k,j})\leq 2w(E_{k,j})$ and we can continue with the sum estimate: 
$$
\sum_{k\geq 0}a^{k\eps}\,  w(\Omega_k)
\leq
2 \sum_{k,j} \left(\frac{1}{|Q_{k,j}|}\int_{Q_{k,j}}w(y)\ud y\right)^{\eps}w(E_{k,j})
$$
$$
\leq
2 \sum_{k,j} 
\int_{E_{k,j}} M_d(w\chi_{Q})^{\eps}\,wdx \leq
2 \int_{Q} M_d(w\chi_{Q})^{\eps}\,wdx
$$

Combining estimates we end up with 
$$
\avgint_{Q}M_d(w \chi_{Q})^{\eps}w
\leq 1+  2\,\eps a^{\eps}\,\log a\, \avgint_{Q} M_d(w\chi_{Q})^{\eps}\,wdx
$$
for any $\eps>0$.  Recall that $a=2^de^{ L\,[w]_{A_\infty}'}  $. Hence if we choose 

I wrote the same steps a bit more compactly:

\begin{equation*}
   L=2^{d+4},\qquad
   \eps = \frac{1}{  2^7 L\,[w]_{A_\infty}'  }
   =    \frac{1}{  2^{11+d}\,[w]_{A_\infty}'  },
\end{equation*}
we can compute
\begin{equation*}
  2\,\eps a^{\eps}\,\log a  <  \frac12,\qquad
  \avgint_{Q}M_d(w \chi_{Q})^{\eps}w\leq 2,
\end{equation*}  
concluding the proof of the theorem.
\end{proof}

\

\subsection{Proof of Theorem  \ref{thm:A1strong(p,p)}, the strong case}

\

The proof is, as in \cite{LOP3}, just an application of  Lemma \ref{keylemma} with a specific parameter~$r$ coming from the 
sharp reverse H\"older inequality given by Theorem \ref{thm:SharpRHI}. Indeed,  since $w\in A_1\subset A_{\infty}$ and if we denote
$$r(w):=1+\frac{1}{  \tau_d\,       [w]_{A_\infty}' },$$ 
we have 
\begin{equation}
  \Big(\avgint_Q w^{r(w)}\Big)^{1/r(w)}\leq 2\avgint_Q w.
\end{equation}
Now by Lemma \ref{keylemma} with $r=r(w)$,  we have 
$$
\|Tf\|_{\strt{1.7ex}L^p(w)}\le c\,pp'\,(r')
^{1/p'}\|f\|_{\strt{1.7ex}L^p(M_rw)} \leq c\,pp'\,([w]_{A_{\infty}}')^{1/p'}\|f\|_{\strt{1.7ex}L^p(2Mw)}
$$
$$
\leq c\,pp'\,([w]_{A_{\infty}}')^{1/p'}\, [w]_{A_1}
^{1/p}\, \|f\|_{\strt{1.7ex}L^p(w)}
$$
using the standard notation \, $M_rw=M(w^r)^{1/r}$. This concludes  the proof of the theorem. 

\

\subsection{Proof of Theorem \ref{thm:A1weak(1,1)}, the weak case}

\

We follow here the classical method of Calder\'on-Zygmund 
with the modifications considered in \cite{perez94}. Applying the
Calder\'on-Zygmund decomposition to $f$ at level\, $\la$, \,  we get
a family of pairwise disjoint cubes $\{Q_j\}$ such that
$$\la<  \frac{1}{|Q_j|} \int_{Q_j} |f|\le 2^d \la$$
Let \, $\Om=\bigcup_jQ_j$ \, and \, $\widetilde \Om=\bigcup_j2Q_j$ \,. The
``good part'' is defined by
$$g=\sum_jf_{Q_j}\chi_{Q_j}(x)+f(x)\chi_{\Om^c}(x)$$
and the ``bad part'' $b$ as
$$b=\sum\limits_{j}b_{j}$$
where
$$b_{j}(x)=(f(x)-f_{Q_{j}})\chi_{\strt{1.7ex}Q_{j}}(x)$$
Then, $f=g+b$.  We split the level set as 
\begin{eqnarray*}
w\{x\in {\mathbb R}^d:|Tf(x)|>\la\}&\le& w(\widetilde \Om)+w\{x\in
(\widetilde \Om)^c:|Tb(x)|>\la/2\}\\
&+&w\{x\in (\widetilde \Om)^c:|Tg(x)|>\la/2\}=I+II+III.
\end{eqnarray*}

Exactly as in  \cite{perez94}, the main term is $III$. We first deal with the easy terms $I$ and $II$, which actually satisfy the better bound
\begin{equation*}
  I+II\lesssim\frac{1}{\lambda}[w]_{A_1}\|f\|_{L^1(w)}.
\end{equation*}
Indeed, the first term 
is essentially the level set of $Mf$:
$$
I= w\{x\in {\mathbb R}^d:Mf(x)>c_d\,\la\}
$$
and the result follows by the classical Fefferman-Stein inequality: 
$$
\big\| Mf \big\|_{L^{1,\infty}(w)} \le c_d\,  \|f\|_{L^{1}(Mw)},
$$
For the second term we use
the following estimate: there is a dimensional constant $c$ such that for any cube $Q$ and any function $b$ supported on $Q$ such that \,$\int_{Q} b(x) \, \ud x=0$ and any weight $w$ we have  
\begin{equation}\label{estimate-atom-like}
\int_{\R^{n} \setminus 2Q  }
\abs{ T b(y) }\, w(y)dy
\le c_d\,
\int_{Q} |b(y)| \, Mw(y)dy  
\end{equation}
This can be found in Lemma
3.3, p.~413, of \cite{GCRdF}.  Now, using this estimate with $w$
replaced by $w\chi_{ \R^{n} \setminus 2Q_{j}   } $
we have
$$
II
\le
\frac{c}{\lambda}
\int_{ \R^{n} \setminus \tilde{\Omega} }
\abs{ Tb (y) }\, w(y)dy 
$$
$$
\leq \frac{c}{\lambda}
\sum_{j}
\int_{ \R^{n} \setminus 2Q_{j} } \abs{Tb_{j}(y)}\,
w(y)dy
\leq 
\frac{c}{\lambda}
\sum_{j}
\int_{ Q_{j} } \abs{b_j(y)}\,
M(w\chi_{_{ \R^{n} \setminus 2 Q_{j}  }})(y)dy
$$
$$
\leq \,\frac{c}{\lambda}
\int_{{\mathbb R}^d}
\abs{ f (y) }\,
Mw(y)dy
+\frac{c}{\lambda}\,
\sum_{j}
\frac{ 1 }{ |Q_{j}|   }
\int_{ Q_{j} }
M( w \chi_{_{ \R^{n} \setminus 2Q_{j}  }} )(x)\, \ud x\, \int_{Q_{j}} \abs{ f(x)} \ud x
$$
Now, to estimate the inner sum we use that \,$M( \chi_{ \R^{n} \setminus 2Q } \mu )$ is essentially constant on $Q$, namely
\begin{equation}\label{MisConstant}
M( \chi_{ \R^{n} \setminus 2Q } \mu )(y)
\approx M( \chi_{_{ \R^{n} \setminus 2Q }} \mu
)(z)   \qquad  y,z \in Q,
\end{equation}
where the constants are dimensional. This fact that can be found in \cite{GCRdF} p.~159. Hence, the sum is controlled by 
$$
c_d\, \sum_{j}  \inf_{x\in Q} \,M(w\chi_{_{ \R^{n} \setminus 2Q_{j}  }})(x)  \int_{Q_{j}} \abs{ f(x)} \ud x
\leq
c_d\, \int_{ \R^{n} } \abs{f(x)}\,  Mw (x) \ud x.
$$
This gives the required estimate.

We now consider last term $III$, the singular term, to which we apply Chebyschev inequality and Lemma~\ref{keylemma} with exponents $p,r\in (1,\infty)$, which will chosen soon,  as follows
\begin{eqnarray*}  
&III &= w\{x\in (\widetilde \Om)^c:|Tg(x)|>\la/2\}\\
&&\le \Norm{T(g)}{L^p(w\chi_{(\widetilde \Om)^c})}^p \\
&&\le c(pp')^p (r')^{\frac{p}{p'}}\,\frac{1}{\la^p}\int_{{\mathbb R}^d}|g|^pM_r(w\chi_{(\widetilde \Om)^c})\ud x\\
&&= c(pp')^p (r')^{p-1}\,\frac{1}{\la}\int_{{\mathbb R}^d}|g|\,M_r(w\chi_{(\widetilde\Om)^c})\ud x.
\end{eqnarray*}
Now, after using the definition of $g$ we use the same argument as above using \eqref{MisConstant} with $M$ replaced by $M_r$. Then we have 
$$\int_{\Om} |g|M_r(w\chi_{(\widetilde \Om)^c})\ud x
\leq
\sum_{j}
\frac{1}{|Q_{j}|} \int_{Q_{j}} \abs{ f(x)}\ud x 
\int_{ Q_{j} }
M_r( w \chi_{_{ \R^{n} \setminus 2Q_{j}  }} )(x)\,\ud x
$$
$$
\leq
c_d\, \sum_{j}  \inf_{x\in Q} \,M_r(w\chi_{_{ \R^{n} \setminus 2Q_{j}  }})(x)  \int_{Q_{j}} \abs{ f(x)}\ud x
\leq
c_d\, \int_{ \Om } \abs{f(x)}\,  M_rw(x)\ud x,
$$
and of course
\begin{equation*}
\int_{\Om^c}|g|M_r(w\chi_{(\tilde\Om)^c})\ud x
\leq\int_{\Om^c}|f|M_r w\ud x.
\end{equation*}

Observe that $r$ is not chosen yet, and we conclude by choosing as above 
the exponent from  Theorem \ref{thm:SharpRHI}
$$r= r(w):=1+\frac{1}{\tau_d\,[w]_{A_{\infty}}' },$$ 
namely the sharp $A_{\infty}$ reverse H\"older's exponent. We also choose 
$$
p=1+\frac{1}{\log(e+[w]_{A_{\infty}}' )}
$$
where $p<2$ and \, $p'\approx \log(e+[w]_{A_{\infty}}')$. Then we continue with
$$w\{x\in (\widetilde \Om)^c:|Tg(x)|>\la/2\}\le
c\,[w]_{A_1} \,  
\log(e+[w]_{A_{\infty}}')  \,\frac{[w]_{ A_{\infty} }'^{2(p-1)}}{\la}   
\int_{{\mathbb R}^d}|f|\,2Mw \ud x.
$$
$$\leq  \frac{c[w]_{A_1}(e+\log[w]_{A_{\infty}}')}{\la}\int_{{\mathbb
R}^d}|f|w \ud x.$$
This estimate combined with the previous ones for $I$ and $II$ completes the proof.

\

\subsection{Proof of Theorem \ref{thm:dualA1weak(1,1)}, the dual weak case}

\

We adapt here the method from \cite{LOP2} where a variant of the Calder\'on-Zygmund decomposition is used, namely  the Calder\'on-Zygmund cubes are replaced by the Withney cubes. Fix $\la>0$, and set
$$\Om_{\la}=\{x\in {\mathbb R}^d:M^c_w(f/w)(x)>\la\},$$
where $M_w^c$ denotes the weighted centered maximal function.  
Let $\bigcup_jQ_j$ be the Whitney covering of $\Om_{\la}$ and set the Calder\'on-Zygmund decomposition $f=g+b$ with respect to these cubes: The
``good part'' is defined by
$$g=\sum_jf_{Q_j}\chi_{Q_j}(x)+f(x)\chi_{\Om^c}(x)$$
and then the  ``bad part'' $b$ is given by 
$$b=\sum\limits_{j}b_{j}$$
where
$$b_{j}(x)=(f(x)-\ave{f}_{Q_{j}})\chi_{\strt{1.7ex}Q_{j}}(x).$$

By the classical Besicovitch lemma we have,
$$w(\Om_{\la})\le \frac{c_n}{\la}  \Norm{f}{L^1( {\mathbb R}^d)}$$
Hence, we have to estimate
\begin{eqnarray*}
w\left\{x\not\in \Om_{\la}:\frac{|Tf(x)|}{w(x)}>\la\right\}&\le&
w\left\{x\not\in \Om_{\la}:\frac{|Tb(x)|}{w(x)}>\la/2\right\}\\&+&
w\left\{x\not\in \Om_{\la}:\frac{|Tg(x)|}{w(x)}>\la/2\right\}\equiv
I_1+I_2.
\end{eqnarray*}

By using again \eqref{estimate-atom-like} with $w=1$,  we obtain
$$
I_1\le \frac{2}{\la}\int_{{\mathbb R}^d\setminus \Om_{\la}}|Tb(x)| \ud x
\le\frac{c}{\la}\sum_j\int_{Q_j}|f-\ave{f}_{Q_j}| \ud x\le
\frac{c}{\la}\Norm{f}{L^1( {\mathbb R}^d)},
$$
where $c=c_{d,T}$.

To estimate $I_2$, we will use the dual version of Lemma \ref{keylemma}, namely:   
\begin{equation}\label{dualkeylemma}
\|Tf\|_{\strt{1.7ex}L^{p'}((M_rw)^{1-p'})}\le cpp'\,(r')
^{1/p'}\|f\|_{\strt{1.7ex}L^{p'}((w)^{1-p'}))}
\end{equation}
As before we  use Theorem \ref{thm:SharpRHI} with
$$r=r(w):=1+\frac{1}{ \tau_d\,[w]_{A_{\infty}}' },$$ 
such that
\begin{equation*}
  \Big(\avgint_Q w^{r}\Big)^{1/r}\leq 2\avgint_Q w .
\end{equation*}
Then \,$M_{r}w\leq 2Mw\leq 2[w]_{A_{1}}w$\, where as usual  \, $M_rw=M(w^r)^{1/r}$. Now, combining Chebyschev inequality with  \eqref{dualkeylemma} with parameter $p \in (1,\infty)$ that will be chosen soon, we have 
\begin{eqnarray*}
I_2
&\le& \frac{2^{p'}}{\la^{p'}}\int_{{\mathbb
R}^d}|Tg|^{p'}w^{1-p'} \ud x\\
&\le& \frac{2^{p'}[w]_{A_1}^{p'-1}}{\la^{p'}}\int_{{\mathbb
R}^d}|Tg|^{p'}M_rw^{1-p'} \ud x\\
&\le&   (cpp')^{p'} \,r'   \frac{[w]_{A_1}^{p'-1}}{\la^{p'}}\int_{{\mathbb
R}^d}|g|^{p'}w^{1-p'} \ud x
\\
&\le&(cp'p)^{p'} \,r' \,   \frac{[w]_{A_1}^{p'-1}}{\la^{p'}}
\left(\int_{{\mathbb
R}^d\setminus{{\Om_{\la}}}}|f|^{p'}w^{1-p'} \ud x
+\sum_{j}(\ave{|f|}_{Q_j})^{p'}\int_{Q_j}w^{1-p'} \ud x\right).
\end{eqnarray*}

We have that $|f|\le \la w$ a.e. on ${\mathbb
R}^d\setminus{{\Om_{\la}}}$, and hence
$$\int_{{\mathbb R}^d\setminus{{\Om_{\la}}}}|f|^{p'}w^{1-p'} \ud x\le
\la^{p'-1}\, \Norm{f}{L^1( {\mathbb R}^d)}.
$$

Next, following again \cite{LOP2}, by properties of the Whitney covering, it is easy to see that
for any cube $Q_j$ there exists a cube $Q_j^*$ such that $Q_j\subset
Q_j^*$, $|Q_j^*|\le c_n|Q_j|$, and the center of $Q_j^*$ lies
outside of $\Om_{\la}$. Therefore,
\begin{eqnarray*}
&&(\ave{|f|}_{Q_j})^{p'-1}\int_{Q_j}w^{1-p'} \ud x\le
[w]_{A_1}^{p'-1}(\ave{|f|}_{Q_j})^{p'-1}\int_{Q_j}(Mw)^{1-p'} \ud x\\
&&\le
[w]_{A_1}^{p'-1}|Q_j|\left(\frac{c\ave{|f|}_{Q_j^*}}{\ave{w}_{Q_j^*}}\right)^{p'-1}\le
(c\la[w]_{A_1})^{p'-1}|Q_j|,
\end{eqnarray*}
which gives
\begin{eqnarray*}
\sum_{j}(\ave{|f|}_{Q_j})^{p'}\int_{Q_j}w^{1-p'} \ud x&\le&
(c\la[w]_{A_1})^{p'-1}\sum_j\ave{|f|}_{Q_j}|Q_j|\\&\le&
(c\la[w]_{A_1})^{p'-1}\, \Norm{f}{L^1( {\mathbb R}^d)}.
\end{eqnarray*}
Combining the previous estimates, recalling that  $r'\approx [w]_{A_{\infty}}'$ we obtain
$$I_2\le c^{p'}\,[w]_{A_{\infty}}'\,p\,(p')^{p'}\frac{p^{p'-1}[w]_{A_1}^{2(p'-1)}}{\la}\,\Norm{f}{L^1( {\mathbb R}^d)},$$
and choosing now $p$ such that 
$$p'=1+\frac{1}{\log(e+[w]_{A_{1}})} \leq 2, $$ 
we get 
$$I_2\le\frac{c[w]_{A_{\infty}}'\log(e+[w]_{A_1})}{\la}\,\Norm{f}{L^1( {\mathbb R}^d)}.$$
This, along with estimates for $I_1$ and for $w(\Om_{\la})$,
completes the proof  of Theorem \ref{thm:dualA1weak(1,1)}.

\section{Commutators,  proof of Theorem \ref{main} and its consequences}  \label{commutator}

For the proof we need a sharp version of the John-Nirenberg Theorem, which can be essentialy found in  \cite{Journe} p. 31--32.

\begin{lemma}[Sharp John-Nirenberg]\label{sharpJ-N}
There are dimensional constants $0\leq \al_d<1<\be_d$ such that
\begin{equation}\label{sharpJN}
\sup_Q \frac{1}{  |Q| } \int_{Q} \exp \left (
\frac{\al_d}{\|b\|_{BMO}} |b(y)- \ave{b}_{Q}| \right  )\, dy \le \be_d.
\end{equation}
In fact we can take $\al_d=\frac{1}{2^{d+2}}$.
\end{lemma}

A key consequence of this lemma for the present purposes is the fact that $e^{\Re zb}w$ inherits the good weight properties of $w$ when the complex number $z$ is small enough. More precisely, for the $A_2$ constant we have:

\begin{lemma}\label{lem:expRezbwA2}
There are dimensional constants $\epsilon_d$ and $c_d$ such that
\begin{equation*}
  [e^{\Re zb}w]_{A_2}\leq c_d[w]_{A_2}\qquad\text{if }|z|\leq\frac{\epsilon_d}{\Norm{b}{BMO}\big([w]_{A_\infty}'+[w^{-1}]_{A_\infty}\big)}.
\end{equation*}
\end{lemma}

\begin{proof}
From the reverse H\"older inequality with exponent $r=1+1/(\tau_d\,[w]_{A_\infty}')$, and the John--Nirenberg inequality, we have for an arbitrary $Q$
\begin{equation*}
\begin{split}
   \avgint_Q  we^{\Re z\,b}
   &\leq\left ( \avgint_Q w^r \right )^{1/r}\left ( \avgint_Q e^{r'\Re z\,(b-\ave{b}_Q) }\right )^{1/r'} e^{\Re z\ave{b}_Q} \\
   &\leq \Big(2\avgint_Q w\Big)\cdot\beta_d\cdot e^{\Re z\ave{b}_Q},\qquad\text{if }\abs{z}\leq\frac{\epsilon_d}{\Norm{b}{\BMO}[w]_{A_\infty}'}.
\end{split}
\end{equation*}
By symmetry, we also have
\begin{equation*}
   \avgint_Q  w^{-1}e^{-\Re z\,b }
   \leq 2\beta_d\Big(\avgint_Q w^{-1}\Big) e^{-\Re z\ave{b}_Q}\qquad\text{if }\abs{z}\leq\frac{\epsilon_d}{\Norm{b}{\BMO}[w^{-1}]_{A_\infty}'}.
\end{equation*}
Multiplication of the two estimates gives
\begin{equation*}
  \Big(\avgint_Q  we^{\Re z\,b}\Big)\Big(\avgint_Q  w^{-1}e^{-\Re z\,b }\Big)
  \leq 4\beta_d^2,
\end{equation*}
for all $z$ as in the assertion, and completes the proof.
\end{proof}

There is an analogous statement for the $A_\infty$ constant $[\ ]_{A_\infty}'$. (A similar result for $[\ ]_{A_\infty}$ is also true, and easier, but we will have no need for it, and it is therefore left as an exercise for the reader.)

\begin{lemma}\label{lem:expRezbwAinfty}
There are dimensional constants $\epsilon_d$ and $c_d$ such that
\begin{equation*}
  [e^{\Re zb}w]_{A_\infty}'\leq c_d[w]_{A_\infty}'\qquad\text{if }|z|\leq\frac{\epsilon_d}{\Norm{b}{BMO}[w]_{A_\infty}'}.
\end{equation*}
\end{lemma}

\begin{proof}
We know that $w$ satisfies the reverse H\"older inequality $\big(\avgint_Q w^{1+3\delta}\big)^{1/(1+3\delta)}\leq 2\avgint_Q w$ with a constant $\delta=c_d/[w]_{A_\infty}'< 2^{-1}$, where $c_d$ is a small dimensional constant. We will prove that $e^{\Re zb}w$ satisfies a reverse H\"older estimate
\begin{equation}\label{eq:expRezbw2prove}
    \Big(\avgint_Q (e^{\Re zb}w)^{1+\delta}\Big)^{1/(1+\delta)}\leq C_d\avgint_Q e^{\Re zb}w,
\end{equation}
for all $z$ as in the assertion.
This shows that 
$$[e^{\Re zb}w]_{A_\infty}'\leq 2C_d/\delta\leq c_d[w]_{A_\infty}'$$
by part b) of Theorem \ref{thm:SharpRHI}.

To prove \eqref{eq:expRezbw2prove}, we first have
\begin{equation*}
\begin{split}
  \Big(\avgint_Q &(e^{\Re zb}w)^{1+\delta}\Big)^{1/(1+\delta)}
  =e^{\Re z\ave{b}_Q}\Big(\avgint_Q(e^{\Re z(b-\ave{b}_Q)}w)^{1+\delta}\Big)^{1/(1+\delta)} \\
  &\leq e^{\Re z\ave{b}_Q}\Big(\avgint_Qe^{\Re z(b-\ave{b}_Q)(1+\delta)^2/\delta}\Big)^{\delta/(1+\delta)^2}
  \Big(\avgint_Q w^{(1+\delta)^2}\Big)^{1/(1+\delta)^2},
\end{split}
\end{equation*}
where we applied H\"older's inequality with exponents  $(1+\delta)/\delta$ and $1+\delta$. Now
\begin{equation*}
  (1+\delta)^2=1+2\delta+\delta^2\leq 1+3\delta,
\end{equation*}
and hence the last factor is bounded by $2\avgint_Q w$. Moreover, by Lemma~\ref{sharpJ-N}, we have
\begin{equation*}
   \avgint_Q e^{\Re z(b-\ave{b}_Q)(1+\delta)^2/\delta}\leq\beta_d \qquad\text{if }\abs{z}\leq\frac{\alpha_d\delta}{4\Norm{b}{BMO}}
\end{equation*}
So altogether
\begin{equation}\label{eq:expRezbwInterm}
  \Big(\avgint_Q (e^{\Re zb}w)^{1+\delta}\Big)^{1/(1+\delta)}
  \leq e^{\Re z\ave{b}_Q}\cdot \beta_d\cdot 2\avgint_Q w,
\end{equation}
and we concentrate on the last factor. We observe that
\begin{equation*}
\begin{split}
  \Big(\avgint_Q w\Big)^2
  &=\Big(\avgint_Q w^{(1+\delta)/2} w^{(1-\delta)/2}\Big)^2 \\
  &\leq\Big(\avgint_Q w^{1+\delta}\Big) \Big(\avgint_Q w^{1-\delta}\Big)
    \leq \Big(2\avgint_Q w\Big)^{1+\delta} \Big(\avgint_Q w^{1-\delta}\Big),
\end{split}
\end{equation*}
and hence
\begin{equation*}
\begin{split}
  \avgint_Q w &\leq 2^{(1+\delta)/(1-\delta)}\Big(\avgint_Q w^{1-\delta}\Big)^{1/(1-\delta)} \\
  &\leq 8\Big(\avgint_Q w^{1-\delta}e^{\Re zb(1-\delta)}e^{-\Re zb(1-\delta)}\Big)^{1/(1-\delta)} \\
  &\leq 8\Big(\avgint_Q w e^{\Re zb}\Big)\Big(\avgint_Q e^{-\Re zb(1-\delta)/\delta}\Big)^{\delta/(1-\delta)},
\end{split}
\end{equation*}
where we used H\"older's inequality with exponents $1/(1-\delta)$ and $1/\delta$.

Combining with \eqref{eq:expRezbwInterm}, we have shown that
\begin{equation*}
\begin{split}
  \Big(\avgint_Q & (e^{\Re zb}w)^{1+\delta}\Big)^{1/(1+\delta)} \\
  &\leq e^{\Re z\ave{b}_Q}\cdot \beta_d\cdot 16\Big(\avgint_Q w e^{\Re zb}\Big)\Big(\avgint_Q e^{-\Re zb(1-\delta)/\delta}\Big)^{\delta/(1-\delta)} \\
  &=16\beta_d\cdot \Big(\avgint_Q w e^{\Re zb}\Big)\Big(\avgint_Q e^{-\Re z(b-\ave{b}_Q)(1-\delta)/\delta}\Big)^{\delta/(1-\delta)} \\
  &\leq 16\beta_d\cdot \Big(\avgint_Q w e^{\Re zb}\Big)\cdot \beta_d,
\end{split}
\end{equation*}
provided that $\abs{z}\leq\alpha_d\delta/\Norm{b}{BMO}$ in the last step. Altogether, we have proven \eqref{eq:expRezbw2prove} with $C_d=16\beta_d^2$, under the condition that $\abs{z}\leq\alpha_d\delta/(4\Norm{b}{BMO})$, and this completes the proof.
\end{proof}

\begin{proof}[Proof of Theorem \ref{main}]  

The proof is a revised version of that of \cite{CPP} following the second proof in the classical $L^p$ theorem for commutators that can be found in \cite{CRW-commutators}. Indeed, we begin by considering the ``conjugate'' of the operator given by 
$$T_z(f)=e^{zb} T(e^{-zb}f).$$
where $z$ is any complex
number. Then, a computation gives (for instance for ``nice'' functions),
$$
[b,T](f)=\frac{d}{dz}T_z(f)|_{z=0}=\frac{1}{2\pi i}\int_{|z|=\ez}
\frac{T_z(f)}{z^2}\,\ud z\, , \qquad \ez>0,$$
by the Cauchy integral theorem.  Now, by Minkowski's inequality
\begin{equation}\label{minkows}
\|[b,T](f)\|_{L^2(w)}\leq \frac{1}{2\pi\,\ez^2} \,\int_{|z|=\ez}
\|T_z(f)\|_{L^2(w)}|\ud z|, \qquad \ez>0,
\end{equation}
all we need to do is estimate \, $
\|T_z(f)\|_{L^2(w)}=\|T(e^{-zb}f)\|_{L^2 (e^{2\Re z\,b }w)} \, ,$
for $|z|=\ez$ with appropriate $\ez$. 
By the main hypothesis of the theorem, we have
\begin{equation*}
  \|T(e^{-zb}f)\|_{L^{2}(w)}
  \le \vp\Big(  [e^{2\Re z\,b }w]_{A_2},[e^{2\Re z\,b }w]_{A_\infty}',[e^{2\Re z\,b }\si]_{A_{\infty}}'\Big)\Norm{e^{-zb}f}{L^2(e^{2\Re z\,b}w)},
\end{equation*}
where $\Norm{e^{-zb}f}{L^2(e^{2\Re z\,b}w)}=\Norm{f}{L^2(w)}$.

By Lemmas~\ref{lem:expRezbwA2} and \ref{lem:expRezbwAinfty} (the latter applied to both $w$ and $w^{-1}$), we have
\begin{equation*}
\begin{split}
  [we^{2\Re bz}]_{A_2}\leq C_d [w]_{A_2},\qquad 
  [we^{2\Re bz}]_{A_\infty}' &\leq C_d[w]_{A_\infty}',\\  [w^{-1}e^{-2\Re bz}]_{A_\infty}' &\leq C_d[w^{-1}]_{A_\infty}',
\end{split}
\end{equation*}
provided that
\begin{equation*}
  \abs{z}=\epsilon\leq\frac{\epsilon_d}{\Norm{b}{BMO}\big([w]_{A_\infty}'+[w^{-1}]_{A_\infty}'\big)}
\end{equation*}
Using this radius and the above estimates in \eqref{minkows}, we obtain
\begin{equation*}
\begin{split}
  \Norm{[b,T](f)}{L^2(w)}
  &\leq\frac{1}{2\pi\epsilon^2}\int_{\abs{z}=\epsilon}\varphi\big(C_d[w]_{A_2},C_d[w]_{A_\infty}',C_d[w^{-1}]_{A_\infty}'\big)\Norm{f}{L^2(w)}\abs{\ud z} \\
  &\leq C_d\Norm{b}{BMO}\big([w]_{A_\infty}'+[w^{-1}]_{A_\infty}'\big) \\
  &\qquad\times\varphi\big(C_d[w]_{A_2},C_d[w]_{A_\infty}',C_d[w^{-1}]_{A_\infty}'\big)\Norm{f}{L^2(w)}.
\end{split}
\end{equation*}

This concludes the proof of the main part of the theorem. The estimate for $T^k_b$ is deduced by iterating from the case $k=1$. 
\end{proof}

\section{Examples} \label{examples}

We compare our new estimates with earlier quantitative results by means of some examples.

\subsection{Power weights and the maximal inequality}

\

Let $d=1$ and $p\in(1,\infty)$ be fixed; we do not pay attention to the dependence of multiplicative constants on $p$.
For $w(x)=\abs{x}^{\alpha}$ and $-1<\alpha<p-1$, one easily checks that
\begin{equation*}
  [w]_{A_p} \eqsim\frac{1}{1+\alpha}\cdot\frac{1}{((p-1)-\alpha)^{p-1}},
\end{equation*}
\begin{equation*}
  [w]_{A_\infty} \eqsim \frac{1}{1+\alpha},\qquad [w^{-1/(p-1)}]_{A_{\infty}}\eqsim\frac{1}{(p-1)-\alpha};
\end{equation*}
moreover, the functionals $[\ ]_{A_\infty}$ and $[\ ]_{A_\infty}'$ are comparable for these weights.

Letting $\alpha\to-1$ or $\alpha\to p-1$, this shows that we have power weights with $[w]_{A_p} = t\gg 1$ and either $[w]_{A_\infty}\eqsim t$ and $[w^{-1/(p-1)}]_{A_\infty}\eqsim 1$, or $[w]_{A_\infty}\eqsim 1$ and $[w^{-1/(p-1)}]_{A_\infty}\eqsim t^{1/(p-1)}$.

With $[w]_{A_p}\eqsim [w]_{A_\infty}\eqsim t\gg 1$ and $[w^{-1/(p-1)}]_{A_\infty}\eqsim 1$, our maximal estimate
\begin{equation*}
  \Norm{M}{\mathscr{B}(L^p(w))}\lesssim\big([w]_{A_p}[w^{-1/(p-1)}]_{A_\infty}\big)^{1/p}\eqsim t^{1/p}
\end{equation*}
clearly improves on Buckley's bound
\begin{equation*}
  \Norm{M}{\mathscr{B}(L^p(w))}\lesssim [w]_{A_p}^{1/(p-1)}\eqsim t^{1/(p-1)}.
\end{equation*}

Despite this improvement over earlier estimates, our bounds fail to provide a two-sided estimate for the norm of the maximal operator:  A.~Lerner and S.~Ombrosi \cite{LO:personal} have constructed a family of weights which shows that
\begin{equation*}
   \inf_{w\in A_2}\frac{\Norm{M}{\mathscr{B}(L^2(w))}}{\big([w]_{A_2}[w^{-1}]_{A_\infty}'\big)^{1/2}}=0.
\end{equation*}
The weights of their example are products of power weights and the two-valued weights considered in the next subsection.

\

\subsection{Two-valued weights and Calder\'on--Zygmund operators}

\

The estimates for the Muckenhoupt constants of power weights in the previous subsection show that
\begin{equation*}
  [w]_{A_2}\eqsim[w]_{A_\infty}+[w^{-1}]_{A_\infty}
  \eqsim[w]_{A_\infty}'+[w^{-1}]_{A_\infty}' 
  \qquad\text{for}\quad w(x)=\abs{x}^{\alpha},\quad d=1,
\end{equation*}
so the improvement of our bound
\begin{equation*}
  \Norm{T }{\mathscr{B}(L^2(w))}\lesssim [w]_{A_2}^{1/2}\big([w]_{A_{\infty}}'+[\si]_{A_{\infty}}'\big)^{1/2}
\end{equation*}
over $\Norm{T }{\mathscr{B}(L^2(w))}\lesssim [w]_{A_2}$ is invisible to such weights.

However, the difference can be observed with weights of the form $w=t\cdot \chi_E+\chi_{\R\setminus E}$, where $t>0$  and $E\subset\R$ is a measurable set, so that both $E$ and $\R\setminus E$ have positive Lebesgue measure. As $I$ ranges over all intervals of $\R$, the ratio $\abs{E\cap I}/\abs{I}$ ranges (at least) over all values $\alpha\in(0,1)$, and hence
\begin{equation*}
  [w]_{A_2}=\sup_{\alpha\in(0,1)}(\alpha t+1-\alpha)(\alpha t^{-1}+1-\alpha)=\frac{(t+1)^2}{4t},
\end{equation*}
and
\begin{equation*}
  [w]_{A_\infty}=\sup_{\alpha\in(0,1)}(\alpha t+1-\alpha)e^{-\alpha\log t}=:\sup_{\alpha\in(0,1)}f(\alpha).
\end{equation*}
Now $f'(\alpha)=0$ at the unique point $\hat\alpha=1/\log t-1/(t-1)\in(0,1)$, and so
\begin{equation*}
  [w]_{A_\infty}=f(\hat\alpha)=e^{-1}\frac{t-1}{\log t}\exp\frac{\log t}{t-1}
  \eqsim\begin{cases} t/\log t, & t\gg 1, \\ t^{-1}/\log t^{-1}, & 0<t\ll 1. \end{cases}
\end{equation*}

Assume then that $t\gg 1$ so that $[w]_{A_{\infty}}\eqsim t/\log t$. Since $\si$ is a weight of the same form with $t^{-1}\ll 1$ in place of $t$, we also have $[\si]_{A_\infty}\eqsim t/\log t$. Thus
\begin{equation*}
  [w]_{A_2}\eqsim t,\qquad \Norm{T}{\mathscr{B}(L^2(w))}\lesssim [w]_{A_2}^{1/2}\big([w]_{A_\infty}+[\si]_{A_\infty}\big)^{1/2}
  \eqsim\frac{t}{\sqrt{\log t}}.
\end{equation*}
In particular, the above estimates already show that
\begin{equation*}
  \inf_{w\in A_2}\frac{\Norm{T}{\mathscr{B}(L^2(w))}}{[w]_{A_2}}=0.
\end{equation*}

\

If we use the sharper version of our $A_2$ theorem with the weight constants $[\ ]_{A_\infty}'$ instead, we find that $\Norm{T}{\mathscr{B}(L^2(w))}$ actually grow much slower than $[w]_{A_2}$:

\begin{lemma}\label{lem:wAinftyPrime}
For $w=t\cdot\chi_E+\chi_{\R\setminus E}$ and $t\geq 3$, we have
\begin{equation*}
  [w]_{A_\infty}'\leq 4\log t.
\end{equation*}
\end{lemma}

(With the earlier estimate for $[w]_{A_\infty}$, this shows that $[w]_{A_\infty}$ can be exponentially larger than $[w]_{A_\infty}'$.)

\begin{proof}
Note that
\begin{equation*}
\begin{split}
  \chi_I M(w\chi_I)
  &=\chi_I \sup_{J\subseteq I}\chi_J\avgint_J w
  =\chi_I \sup_{J\subseteq I}\chi_J\frac{1}{\abs{J}}(\abs{J\setminus E}+t\abs{J\cap E}) \\
  &=\chi_I \sup_{J\subseteq I}\chi_J\big(1+(t-1)\frac{\abs{J\cap E}}{\abs{J}}\big)
    =\chi_I\big(1+(t-1) M(\chi_{I\cap E})\big),
\end{split}
\end{equation*}
and hence, abbreviating $\tau:=t-1$,
\begin{equation*}
\begin{split}
  \int_I M(w\chi_I)
  &=\abs{I}+\tau \int_I M(\chi_{I\cap E})
  =\abs{I}+\tau \int_0^1\abs{I\cap\{ M(\chi_{I\cap E})>\lambda\}}\ud\lambda \\
  &\leq\abs{I}+\tau\Big(\int_0^a\abs{I}\ud\lambda+\int_a^1\frac{2}{\lambda}\abs{I\cap E}\ud\lambda\Big) \\
  &=\abs{I}+\tau\Big(a\abs{I}+2\abs{I\cap E}\log\frac{1}{a}\Big) \\
  &=\abs{I}+\tau\abs{I\cap E}\Big(1+2\log\frac{\abs{I}}{\abs{I\cap E}}\Big),\qquad a:=\frac{\abs{I\cap E}}{\abs{I}},
\end{split}
\end{equation*}
where the factor $2$ is the weak-type $(1,1)$ norm of the maximal operator on the real line.
Since $w(I)=\abs{I}+\tau\abs{I\cap E}$, we have
\begin{equation}\label{eq:wAinftyPrime}
\begin{split}
  [w]_{A_\infty}'
  =\sup_I\frac{1}{w(I)}\int_I M(w\chi_I)
  &\leq\sup_{\alpha\in(0,1)}\frac{1+\tau\alpha(1+2\log\alpha^{-1})}{1+\tau\alpha} \\
  &=1+2\sup_{\alpha\in(0,1)}\frac{\tau\alpha}{1+\tau\alpha}\log\frac{1}{\alpha}, \\
\end{split}
\end{equation}
recalling that the ratio $\abs{I\cap E}/\abs{I}$ attains at least all values $\alpha\in(0,1)$ as $I$ ranges over all intervals.

If $\alpha\geq\tau^{-1}$, then $\log\alpha^{-1}\leq\log\tau$, while $\tau\alpha/(1+\tau\alpha)\leq 1$. If $\alpha\leq\tau^{-1}$, then
\begin{equation*}
  \tau\alpha\log\frac{1}{\alpha}=\tau\alpha\log\frac{1}{\tau\alpha}+\tau\alpha\log\tau\leq\frac{1}{e}+\log\tau,
\end{equation*}
as $x\log x^{-1}\leq e^{-1}$ and $x\leq 1$ for $x=\tau\alpha\in(0,1)$. Altogether, recalling that $t=\tau+1\geq 3$, we have
\begin{equation*}
 [w]_{A_\infty}'\leq
 1+2\big(\frac{1}{e}+\log\tau)\leq(1+\frac{2}{e})+2\log t\leq 4\log t.\qedhere
\end{equation*}
\end{proof}

Since $\sigma=w^{-1}$ is a weight of the same form, we find that for these particular weights,
\begin{equation*}
  [w]_{A_2}\eqsim t,\qquad\Norm{T}{\mathscr{B}(L^2(w))}\lesssim [w]_{A_2}^{1/2}\big([w]_{A_\infty}'+[\sigma]_{A_\infty}'\big)^{1/2}
  \lesssim \big(t\log t)^{1/2},
\end{equation*}
so indeed $\Norm{T}{\mathscr{B}(L^2(w))}$ can grow much slower than $[w]_{A_2}$ for such particular families of weights. This example also motivates the use of the $A_\infty$ constants $[w]_{A_\infty}'$, rather than $[w]_{A_\infty}$, whenever this is possible.

\

In a similar way we can show that the main  result from Theorem \ref{thm:A1strong(p,p)} strictly improves on the earlier estimate \eqref{A1strong(p,p)}. Indeed, if we let $w$ be the previous weight with  $t\gg 1$ so that $w_{A_1}\eqsim t$ and $[w]_{A_{\infty}}\eqsim t/\log t$, then
$$
 [w]_{A_1}^{1/p} \, [w]_{A_{\infty}}^{1/p'} \eqsim \frac{t}{ (\log t)^{1/p'} }.
$$
As above, this family of weights shows that
\begin{equation*}
\inf_{w\in A_1}\frac{\Norm{T}{\mathscr{B}(L^p(w))}}{[w]_{A_1}}=0, \qquad 1<p<\infty.
\end{equation*}

\

\subsection{Two-valued weights and dyadic shifts}

\

Although it was not stated explicitly above, from the proof it is clear that our weighted bound for the dyadic shifts only depends on the dyadic Muckenhoupt constants, where the supremum is over dyadic cubes only, instead of all cubes. This makes a difference for the two-valued weights $w=t\cdot \chi_E+\chi_{\R\setminus E}$ considered above, when the set $E$ is appropriately chosen. Indeed, with $E:=\bigcup_{k\in\Z}[2k,2k+1)$, one observes that the ratio $\abs{E\cap I}/\abs{I}$ only attains the values $0,\tfrac12,1$ as $I$ ranges over the dyadic intervals. Consequently, the dyadic $A_\infty$ constant has a different expression
\begin{equation*}
  [w]_{A_\infty}^d=\max_{\alpha\in\{0,\tfrac12,1\}}(\alpha t+1-\alpha)e^{-\alpha t}=\frac{t+1}{2\sqrt{t}}
  =([w]_{A_2}^d)^{1/2},
\end{equation*}
where $[w]_{A_2}^d=[w]_{A_2}$, as one easily observes. 
Repeating the proof of Lemma~\ref{lem:wAinftyPrime} in the dyadic case (recalling that the weak-type $(1,1)$ norm is $C_d=1$ for the dyadic maximal operator), we get in place of \eqref{eq:wAinftyPrime} that
\begin{equation*}
  [w]_{A_\infty}^{\prime,d}\leq 1+\sup_{\alpha\in\{0,1/2,1\}}\frac{\tau\alpha}{1+\tau\alpha}\log\frac{1}{\alpha}
  =1+\frac{\tfrac12\tau}{1+\tfrac12\tau}\log 2\leq 1+\log 2.
\end{equation*}
So these constants are actually uniformly bounded over the choice of the parameter~$t$.

By symmetry, we also have $[w^{-1}]_{A_\infty}^d=[w]_{A_\infty}^d$ and $[w^{-1}]_{A_\infty}^{\prime,d}=[w]_{A_\infty}^{\prime,d}$, and hence, for this particular $E$ and $w=t\cdot \chi_E+\chi_{\R\setminus E}$,
\begin{equation*}
  \Norm{\sha}{\mathscr{B}(L^2(w))}\lesssim (r+1)^2\big([w]_{A_2}^{d}\big)^{1/2}\big([w]_{A_\infty}^{\prime,d}+[w^{-1}]_{A_\infty}^{\prime,d}\big)^{1/2}
  \lesssim(r+1)^2\big([w]_{A_2}^d\big)^{1/2}.
\end{equation*}
The first $A_\infty$ constants $[\ ]_{A_\infty}^d$ would have given the weaker bound $\Norm{\sha}{\mathscr{B}(L^2(w))}\lesssim(r+1)^2\big([w]_{A_2}^d\big)^{3/4}$, instead.

\

\subsection{The extrapolated bounds for Calder\'on--Zygmund operators}

\

It is interesting to compare our estimate \eqref{cor:CZupper}, namely
\begin{equation}\label{eq:CZupperRepeat}
\begin{split}
    \Norm{T}{\mathscr{B}(L^p(w))}
  &\lesssim[w]_{A_p}^{2/p-1/[2(p-1)]}\big([w]_{A_\infty}^{1/[2(p-1)]}+[\si]_{A_\infty}^{1/2}\big)([w]_{A_\infty}')^{1-2/p} \\
  &\lesssim[w]_{A_p}^{2/p}([w]_{A_\infty}')^{1-2/p}, \\
\end{split}
\end{equation}
which is valid for any Calder\'on--Zygmund operator and for all $p\geq 2$, with an estimate implicitly contained in the proof of a related result by Lerner \cite{Lerner:Ap}, Theorem~1.2. He considers maximal trunctations $T_*$ of convolution-type Calder\'on--Zygmund operators, and obtains the following bound:
\begin{equation}\label{eq:LernerAinfty}
\begin{split}
  \Norm{T_*}{\mathscr{B}(L^p(w))}
  &\lesssim[w]_{A_p}^{1/2}([w]_{A_\infty}')^{1/2}+\Norm{M}{\mathscr{B}(L^p(w))},\\
  &\lesssim[w]_{A_p}^{1/2}([w]_{A_\infty}')^{1/2},\qquad p\in[3,\infty),
\end{split}
\end{equation}
where the second estimate is an application of Buckley's result (we do not even need our improvement at this point),
\begin{equation*}
  \Norm{M}{\mathscr{B}(L^p(w))}\lesssim[w]_{A_p}^{1/(p-1)}\leq[w]_{A_p}^{1/2},\qquad p\in[3,\infty),
\end{equation*}
In \eqref{eq:LernerAinfty},  the factor $([w]_{A_\infty}')^{1/2}$ comes from an estimate of Wilson \cite{Wilson:89} relating the weighted norms of the grand maximal function and a certain square function, while $[w]_{A_p}^{1/2}$ is Lerner's  bound for the weighted norm of such square functions (whose exponent is optimal by \cite{CMP}).

To simplify comparison, let us only consider the simpler form of our bound \eqref{eq:CZupperRepeat}. Then the sum of  the powers of $[w]_{A_p}$ and $[w]_{A_\infty}'$ in both \eqref{eq:CZupperRepeat} and \eqref{eq:LernerAinfty} is $2/p+(1-2/p)=1/2+1/2=1$, and the sharper bound is the one where the larger weight constant $[w]_{A_p}$ has the smaller power. We have $2/p\leq 1/2$ if and only if $p\geq 4$, and hence Lerner's bound is sharper for $p\in[3,4)$ and ours for $p\in(4,\infty)$.
This indicates that the present results might not be the last word on joint $A_p$--$A_\infty$-control, but there is place for further investigation.

\section{Proof of the end-point theory at $p=\infty$}

The proof again relies on the sharp reverse H\"older inequality Theorem \ref{thm:SharpRHI}: if \,$w \in A_{\infty}$ and if we let 
$$r=r(w):=1+\frac{1}{c_d\,[w]_{A_{\infty}}'},$$ 
then
$$
\left (  \avgint_Qw^{r} \ud x\right )^{\frac{1}{r}} \le
\frac{2}{|Q|}\int_Qw\,. $$

\begin{proof}[Proof of Theorem~\ref{thm:embNorm}]
For $c=\ave{f}_Q$,
\begin{equation*}
\begin{split}
  \frac{1}{w(Q)}\int_Q\abs{f-c}w
  &=\frac{\abs{Q}}{w(Q)}\avgint_Q\abs{f-c}w \\
  &\leq\frac{\abs{Q}}{w(Q)}\Big(\avgint_Q\abs{f-c}^{r(w)'}\Big)^{1/r(w)'}\Big(\avgint_Q w^{r(w)}\Big)^{1/r(w)} \\
  &\leq\frac{\abs{Q}}{w(Q)} \Big(C_d r(w)'\Norm{f}{\BMO}\Big)\Big(2\avgint_Q w\Big) \\
  &= C_d r(w)'\Norm{f}{\BMO}\leq C_d[w]_{A_{\infty}}'\Norm{f}{\BMO},
\end{split}
\end{equation*}
which shows that $\Norm{f}{\BMO(w)}\leq C_d[w]_{A_{\infty}}'\Norm{f}{\BMO}$. Note that we used the sharp order of growth of the local $L^p$ norms of $\BMO$ functions as $p\to\infty$, which follows easily from the exponential integrability.

To see the sharpness for $d=1$, consider $w(x)=\abs{x}^{-1+\eps}$, which has $[w]_{A_{\infty}}\eqsim[w]_{A_\infty}'\eqsim 1/\eps$ and $f(x)=\log\abs{x}$. We check that
\begin{equation*}
  \Norm{f}{\BMO(w)}\geq\inf_a\frac{1}{w([0,1])}\int_0^1\abs{\log\frac1x-a}w(x)\ud x\geq\frac{c}{\eps}\geq c[w]_{A_{\infty}}\geq c[w]_{A_\infty}',
\end{equation*}
which proves the claim. It is immediate that $w([0,1])=\int_0^1 x^{-1+\eps}\ud x=1/\eps$. It remains to compute
\begin{equation*}
  \int_0^1\abs{\log\frac1x-a}x^{-1+\eps}\ud x
  =\int_0^{\infty}\abs{t-a}e^{-\eps t}\ud t
  =\frac{1}{\eps^2}\int_0^{\infty}\abs{u-\eps a}e^{-u}\ud u.
\end{equation*}
It suffices to check that $\psi(\alpha):=\int_0^{\infty}\abs{u-\alpha}e^{-u}\ud u\geq c>0$ for all $\alpha\in\R$. But this is an easy calculus exercise.
\end{proof}

We now prove Corollary~\ref{cor:LinftyBMO} on end-point estimates for Calder\'on--Zygmund operators.

\begin{proof}[Proof of Corollary~\ref{cor:LinftyBMO}]
For the positive estimate, it suffices to factorize $T=I\circ T$, where $T:L^{\infty}\to\BMO$ and $I:\BMO\to\BMO(w)$ have norm bounds $c_{T}$ and $c_d[w]_{A_{\infty}}'$, respectively. Concerning sharpness, note that the Hilbert transform of $\chi_{(-1,0)}$ is $\log(x+1)-\log x$ for $x>0$. Since $\log(x+1)$ is bounded on $[0,1]$, the computation proving the sharpness of the embedding $\BMO\hookrightarrow\BMO(w)$ also gives the lower bound
\begin{equation*}
\begin{split}
 \Norm{H\chi_{(-1,0)}}{\BMO(\abs{x}^{-1+\eps})}\geq c/\eps
 &=c[x^{-1+\eps}]_{A_{\infty}}\Norm{\chi_{(-1,0)}}{L^{\infty}} \\
 &\geq c[x^{-1+\eps}]_{A_{\infty}}'\Norm{\chi_{(-1,0)}}{L^{\infty}}.\qedhere
\end{split}
\end{equation*}
\end{proof}

We conclude with the proof of Proposition~\ref{BMO-Ainfty} on the sharp relation of $A_\infty$ and $\BMO$. Note that here we use the larger constant $[w]_{A_\infty}$, not $[w]_{A_\infty}'$.

\begin{proof}[Proof of Proposition~\ref{BMO-Ainfty}]

Let $Q$ be a cube. We estimate
\begin{equation*}
\begin{split}
  \int_Q\abs{\log w-\log c}
  &=\int_{Q\cap\{w\geq c\}}\log\frac{w}{c}+\int_{Q\cap\{w<c\}}\log\frac{c}{w} \\
  &=\int_{Q\cap\{w\geq c\}}\log\frac{w}{c}+\Big(\int_Q-\int_{Q\cap\{w\geq c\}}\Big)\log\frac{c}{w} \\
  &=2\int_{Q\cap\{w\geq c\}}\log\frac{w}{c}+\int_Q\log c+\int_Q\log\frac{1}{w} \\
  &\leq 2\int_{Q\cap\{w\geq c\}}\frac{w}{c}+\abs{Q}\log c+\abs{Q}\log\Big([w]_{A_\infty}\Big/\avgint_Q w\Big).
\end{split}
\end{equation*}
Hence
\begin{equation*}
  \avgint_Q \abs{\log w-\log c}
  \leq \frac{2}{c} \avgint_Q w+\log c+\log[w]_{A_\infty}-\log\Big( \avgint_Q w\Big).
\end{equation*}
Choosing $c=c_Q=2\avgint_Q w$, we get
\begin{equation*}
   \avgint_Q \abs{\log w-\log c_Q}
  \leq 1+\log 2+\log\Big( \avgint_Q w\Big)+\log[w]_{A_\infty}-\log\Big( \avgint_Q w\Big)=\log(2e[w]_{A_\infty}),
\end{equation*}
and this proves that
\begin{equation*}
 \Norm{\log w}{BMO}\leq\log(2e[w]_{A_\infty}).\qedhere
\end{equation*}  
\end{proof}

\begin{remark}
In the last estimate, we cannot replace $[w]_{A_\infty}$ by $[w]_{A_\infty}'$. Indeed, for the two-valued weight $w=t\cdot 1_E+1_{\R\setminus E}$, one readily checks that $\Norm{\log w}{BMO}\eqsim\log t$, whereas Lemma~\ref{lem:wAinftyPrime} shows that also $[w]_{A_\infty}'\lesssim\log t$. Thus $\Norm{\log w}{BMO}\leq\log(c[w]_{A_\infty}')$ would lead to the obvious contradiction that $\log t\leq c+\log\log t$.
\end{remark}

\appendix

\bibliography{weighted}

\bibliographystyle{plain}

\end{document}